\documentclass[a4paper, 11pt]{amsart}
\usepackage[utf8]{inputenc}
\usepackage[T1]{fontenc}
\usepackage[english]{babel}
\usepackage[margin=1in]{geometry} 
\usepackage{microtype, xcolor, graphicx}
\usepackage{amsmath, amssymb, amsthm}
\usepackage[linktocpage=true]{hyperref}
\usepackage[capitalise]{cleveref}
\usepackage{enumitem}
\usepackage{tikz-cd} 

\theoremstyle{plain}
\newtheorem{mainthm}{Theorem}

\newtheorem{thm}{Theorem}[section]
\newtheorem{lem}[thm]{Lemma}
\newtheorem{prop}[thm]{Proposition}
\newtheorem{cor}[thm]{Corollary}
\newtheorem*{prop*}{Proposition}

\theoremstyle{definition}
\newtheorem{defn}[thm]{Definition}

\newtheorem{exmp}[thm]{Example}

\theoremstyle{remark}
\newtheorem{rem}[thm]{Remark}

\numberwithin{equation}{section}

\DeclareMathOperator{\Mods}{Mod}
\DeclareMathOperator{\mods}{mod}
\DeclareMathOperator{\Hom}{Hom}
\DeclareMathOperator{\RHom}{\mathbf{R}Hom}
\DeclareMathOperator{\otimesL}{\otimes_R^\mathbf{L}}
\DeclareMathOperator{\Ext}{Ext}
\DeclareMathOperator{\Fun}{Fun}
\DeclareMathOperator{\id}{id}
\DeclareMathOperator{\Spec}{Spec}
\DeclareMathOperator{\Supp}{Supp}
\DeclareMathOperator{\supp}{supp}
\DeclareMathOperator{\coker}{coker}
\DeclareMathOperator{\im}{im}
\DeclareMathOperator{\cone}{cone}
\DeclareMathOperator{\colim}{colim}
\DeclareMathOperator{\hocolim}{\underrightarrow{\mathrm{hocolim}}}
\DeclareMathOperator{\rank}{rank}

\DeclareMathOperator{\Flat}{Flat}

\DeclareMathOperator{\ind}{ind}

\DeclareMathOperator{\add}{\mathsf{add}}

\DeclareMathOperator{\gen}{\mathsf{gen}}
\DeclareMathOperator{\wide}{\mathsf{wide}}
\DeclareMathOperator{\thick}{\mathsf{thick}} 
\DeclareMathOperator{\loc}{\mathsf{loc}}
\DeclareMathOperator{\susp}{\mathsf{susp}}
\DeclareMathOperator{\cosusp}{\mathsf{cosusp}}
\DeclareMathOperator{\Aisle}{Aisle}
\DeclareMathOperator{\aisle}{\mathsf{aisle}}
\newcommand{\cD}{\mathcal{D}} 
\newcommand{\cT}{\mathcal{T}}
\newcommand{\cK}{\mathcal{K}}
\newcommand{\cA}{\mathcal{A}} 

\newcommand{\cH}{\mathcal{H}}
\newcommand{\cS}{\mathcal{S}}
\newcommand{\cU}{\mathcal{U}} 
\newcommand{\cX}{\mathcal{X}}
\newcommand{\cV}{\mathcal{V}} 
\newcommand{\cW}{\mathcal{W}}
\newcommand{\cY}{\mathcal{Y}}

\renewcommand{\a}{\mathfrak{a}} 
\renewcommand{\b}{\mathfrak{b}}
\newcommand{\p}{\mathfrak{p}}
\newcommand{\q}{\mathfrak{q}}

\newcommand{\bN}{\mathbb{N}} 
\newcommand{\bZ}{\mathbb{Z}}
\newcommand{\bK}{\mathbb{K}} 
\newcommand{\kp}{\kappa(\p)}
\newcommand{\hs}{\mathbf{hs}}

\newcommand{\longmapsfrom}{\mathrel{\reflectbox{\ensuremath{\longmapsto}}}}
\newcommand{\leftangle}{\left\langle}
\newcommand{\rightangle}{\right\rangle}
\newcommand{\leftcurly}{\left\{}
\newcommand{\rightcurly}{\right\}}

\begin{document}
\title[Telescope conjecture for t-structures over noetherian path algebras]{Telescope conjecture for t-structures\\over noetherian path algebras}
\author{Enrico Sabatini}
\address{Enrico Sabatini, Dipartimento di Matematica ``Tullio Levi-Civita'', Universit{\`a} degli Studi di Padova, via Trieste 63, 35121 Padova, Italy}
\email{enrico.sabatini@phd.unipd.it}
\thanks{The author was supported by the Department of Mathematics ``Tullio Levi-Civita'' of the University of Padova during his PhD, and by the BIRD - Budget Integrato per la Ricerca dei Dipartimenti 2022, within the project \textit{Representations of quivers with commutative coefficients}, for the participation at the Network meeting on the telescope conjecture.}
\begin{abstract}
    Let $RQ$ be the path algebra of a Dynkin quiver $Q$ over a commutative noetherian ring $R$. We show that any homotopically smashing t-structure in the derived category of $RQ$ is compactly generated. We also give a complete description of the compactly generated t-structures in terms of poset homomorphisms from the prime spectrum of the ring $\Spec(R)$ to the poset of filtrations of noncrossing partitions of the quiver $\mathrm{Filt}(\mathbf{Nc}(Q))$. In the case that $R$ is regular, we also get a complete description of the wide subcategories of the category $\mods(RQ)$.
\end{abstract}
\maketitle
\setcounter{tocdepth}{1}
\tableofcontents

\section*{Introduction}
Given a compactly generated triangulated category, the telescope conjecture asks whether every smashing subcategory is compactly generated. In many cases, affirmative answers to this question are accompanied by a full classification of the relevant subcategories.

In the setting of derived categories of rings, a landmark result is due to Neeman \cite{NeeTC}, who proved that for any commutative noetherian ring $R$, the smashing subcategories of $\cD(R)$ are in bijection with the specialization-closed subsets of $\Spec(R)$, and that the telescope conjecture holds in this context. Other positive results for derived categories of rings appear in \cite{DPTC, KSTC, SteTC}, a criterion for rings of weak global dimension at most one is given in \cite{BSTC}, while a well-known counterexample is provided in \cite{KelTC}. Many of these (counter)examples are elegantly unified under the general framework developed in \cite{HTC}.

In the context of representations of Dynkin quivers over commutative noetherian rings, the work of Antieau and Stevenson \cite{AS} extends Neeman’s result. They established a bijection between the lattice of smashing subcategories of $\cD(RQ)$ and the lattice of poset homomorphisms from $\Spec(R)$ to the lattice of noncrossing partitions $\mathbf{Nc}(Q)$, and proved that the telescope conjecture continues to hold in this broader setting.

More recently, a generalization of the telescope conjecture has been proposed in the context of t-structures that are not necessarily stable (i.e. arising from triangulated subcategories) \cite[Question A]{BHTC} and \cite[Question A.7]{HN}. In this setting, the conjecture asks whether every homotopically smashing t-structure (in the sense of \cite{SSV}) is compactly generated. For stable t-structures, this condition coincides with requiring the aisle to be smashing, thereby recovering the classical formulation.

In the derived category $\cD(R)$ of a commutative noetherian ring, Alonso, Jerem{\'i}as, and Saor{\'i}n \cite{AJS} provided a full classification of compactly generated t-structures in terms of filtrations (i.e. non increasing sequences) of specialization closed subsets of $\Spec(R)$. Subsequently, Hrbek and Nakamura \cite{HN} showed that the generalized telescope conjecture holds in this setting. Additional examples proving the conjecture for t-structures appear in \cite{AHTC, BHTC}.

In this paper, we extend \cite{AJS} and \cite{HN} to the derived category of representations of Dynkin quivers over commutative noetherian rings. We classify all compactly generated (aisles of) t-structures in $\cD(RQ)$ and prove that the telescope conjecture holds, unifying and generalizing the aforementioned results. Let $\mathrm{Filt}(\mathbf{Nc}(Q))$ be the lattice of filtrations of noncrossing partitions of $Q$, then our first main theorem is:

\begin{mainthm}[\cref{Main}]\label{ThmA}
    Let $R$ be a commutative noetherian ring and $Q$ a Dynkin quiver. Then:
    \begin{enumerate}
        \item There is an order-preserving bijection
        \[\begin{tikzcd} \Aisle_\mathrm{cg}(\cD(RQ)) \arrow[r, leftrightarrow] & \Hom_\mathrm{Pos}(\Spec(R),\mathrm{Filt}(\mathbf{Nc}(Q)))\end{tikzcd}\]
        \item Any homotopically smashing t-structure in $\cD(RQ)$ is compactly generated.
    \end{enumerate} 
\end{mainthm}

\noindent The bijection in (1) is obtained by identifying $\mathrm{Filt}(\mathbf{Nc}(Q))$ with the lattice of compactly generated aisles in $\cD(\bK Q)$, where $\bK$ is an arbitrary field (\cref{FiltNc}), and interpreting poset homomorphisms as order-preserving maps $\sigma$ for which $\sigma(\p)\in\Aisle(\cD(\kp Q))$ for each $\p\in\Spec(R)$. In this framework, the two assignments are realized by ``projecting'' a compactly generated aisle of $\cD(RQ)$ to $\cD(\kp Q)$ and by ``gluing'' along $\Spec(R)$ the aisles in the image of a homomorphism.

To fully understand the classified t-structures, we also characterize the homotopically smashing coaisles. This leads to the observation that compactly generated aisles are determined on cohomology (\cref{DetCoh}). Applying a result of Zhang and Cai \cite{ZC}, we then obtain a complete classification of wide subcategories of $\mods(RQ)$, generalizing both Takahashi’s classification over commutative noetherian rings \cite{Tak} and the work of Ingalls and Thomas over Dynkin algebras \cite{IT}. This also completes the picture outlined in \cite{IK}, where similar results are obtained for torsion classes and Serre subcategories of $\mods(RQ)$. Our second main result is:

\begin{mainthm}[\cref{Wide}]\label{ThmB}
    Let $R$ be a commutative noetherian regular ring and $Q$ a Dynkin quiver. Then there is an order-preserving bijection:
    \[\begin{tikzcd} \mathrm{Wide}(RQ) \arrow[r, leftrightarrow] & \Hom_\mathrm{Pos}(\Spec(R),\mathbf{Nc}(Q)) \end{tikzcd}\]
\end{mainthm}

\noindent Once again, the poset of noncrossing partitions $\mathbf{Nc}(Q)$ is interpreted as the lattice of wide subcategories of $\mods(\bK Q)$, independent of the choice of field, and the assignments correspond to suitable ``projection'' and ``gluing'' operations.

\subsection*{Structure of the paper}
The paper is divided into four sections and an appendix. In \textbf{\cref{Prelim}}, we recall basic and preliminary results on representations of small categories $C$ over commutative rings $R$, focusing in particular on the action of $\cD(R)$ on $\cD(RC)$, the effect of this action on t-structures, and the support theory of $\cD(RC)$. In \textbf{\cref{LTGM}}, we investigate two key properties -- the \textit{local to global principle} and the \textit{minimality of stalk subcategories} -- which are already known in the context of stable t-structures, and we prove them for general t-structures in $\cD(RC)$. In \textbf{\cref{DynLift}}, we specialize to representations of Dynkin quivers and lay the groundwork for the classification results, studying compactly generated aisles, homotopically smashing coaisles, and the relation between them. In \textbf{\cref{MainRes}}, we apply the preceding results to prove \cref{ThmA}, and devote the final part of the section to the study of wide subcategories and the proof of \cref{ThmB}. Finally, in \textbf{Appendix} \ref{App}, we present a complete proof that the t-structures of a Dynkin algebras (i.e. a finite-dimensional hereditary algebra of finite representation type) are parametrized by filtrations of noncrossing partitions of the quiver (\cref{FiltNc}), and are therefore independent of the field; while this result may be familiar to specialists, to our knowledge it has not been explicitly recorded in the literature.

\subsection*{Notation and conventions}
In a category $\cT$, we will often consider the smallest subcategory containing a class of objects $\cX$ of $\cT$ and closed under certain operations. We denote such subcategories by Sans-serif letters and angle brackets; for example, when $\cT$ is triangulated with coproducts, we write $\loc_\cT\langle\cX\rangle$ for the smallest triangulated subcategory of $\cT$ containing $\cX$ and closed under coproducts (i.e. the smallest localizing subcategory of $\cT$ containing $\cX$). If instead we consider the poset of all subcategories of a certain type in $\cT$, ordered by inclusion, we denote it by Roman letters and round brackets; in the previous example, $\mathrm{Loc}(\cT)$ denotes the lattice of localizing subcategories of $\cT$.

We also consider subcategories defined by orthogonality conditions. For two classes of objects $\cX$ and $\cY$ in $\cT$, we write:
\[\Hom_\cT(\cX,\cY)=\{f\in\Hom_\cT(X,Y)\mid X\in\cX,\,Y\in\cY\}\]
\[\cX^\perp=\{Y\in\cT\mid\Hom_\cT(\cX,Y)=0\} \text{ and } {}^\perp\cY=\{X\in\cT\mid\Hom_\cT(X,\cY)=0\}\]

For a derived category $\cD(\cA)$ of an abelian category $\cA$, a class of objects $\cS$ in $\cA$, and a set of integers $I\subseteq\bZ$, we write $\cD^I(\cS)$ for the full subcategory of $\cD(\cA)$ consisting of complexes $X$ such that $X^i$ belongs to $\cS$ if $i\in I$, and is zero otherwise. We write $\cD^{\geq0}(\cS)$ for $\cD^{\{i\geq0\}}(\cS)$, and similarly for other intervals.

\subsection*{Acknowledgements}
The author would like to express his immense gratitude to his supervisors Jorge Vit{\'o}ria and Michal Hrbek. In particular, to the former for being a patient and helpful guide, both on a human and mathematical level, fundamental to the realization of this paper; to the latter for his hospitality during the author's visit to the Czech Academy of Sciences and for pointing out useful arguments concerning the proofs of \cref{Min} and \cref{CohDetThm}; to both for carefully reading a preliminary version of this paper and for providing useful feedback. The author would also like to thank Greg Stevenson and {\'A}lvaro S{\'a}nchez Campillo for several discussions on the topics of the paper.

\section{Preliminaries}\label{Prelim}
\subsection{Representations of small categories}
Given a small category $C$ and a ring $R$, the category of left $C$-modules over $R$ is the category of functors $\Mods_R(C)=\Fun(C,\Mods(R))$. An equivalent description, as the category of $R$-linear functors $\Fun_R(RC,\Mods(R))$ is given in \cite[Lemma 2.7]{AS}, where $RC$ is the $R$-linearization of the small category $C$, i.e. the category with the same objects as $C$ and with hom-set $\Hom_{RC}(c,c')$ equal to the free $R$-module with basis $\Hom_C(c,c')$, for any two objects $c,c'\in C$.

\begin{lem}[{\cite[III 4.2, 4.6, 5.2]{Pop}}]\label{GroCat}
    The functor category $\Mods_R(RC)$ is a Grothendieck abelian category with exact products and a set of small projective generators. In particular, it has also exact coproducts, an injective cogenerator and a projective generator.
\end{lem}

As for any Grothendieck abelian category, we can construct the homotopy category $\cK(RC)$ and the derived category $\cD(RC)$. Recall that a triangulated category with coproducts $\cT$ is said to be \textit{compactly generated} if the subcategory $\cT^c$ formed by its \textit{compact objects}, i.e. the objects $X\in\cT$ such that $\Hom_\cT(X,\_)$ commutes with coproducts, is skeletally small and $\loc(\cT^c)=\cT$. 

\begin{rem}\label{DerCat}
    The derived category $\cD(RC)$ is compactly generated. Indeed, the set $\mathcal{P}$ of all shifts of the small projective generators of $\Mods_R(RC)$ is a set of compact objects in $\cD(RC)$ such that $\mathcal{P}^\perp=0$. In particular, $\loc\langle\cD^c(RC)\rangle=\cD(RC)$ and $\thick\langle\mathcal{P}\rangle=\cD^c(RC)$.
\end{rem}

For an abelian category $\cA$, we will call a complex $I\in\cK(\cA)$ \textit{K-injective} if $\Hom_{\cK(\cA)}(A,I)=0$ for every acyclic complex $A\in\cK(\cA)$; we define dually \textit{K-projective} complexes.

\begin{lem}\label{Res}
    The homotopy category $\cK(RC)$ admits K-injective and K-projective resolutions, i.e. for any complex $X\in\cK(RC)$ there are quasi-isomorphisms $X\to\mathbf{i}X$ and $\mathbf{p}X\to X$, where $\mathbf{i}X$ and $\mathbf{p}X$ are a K-injective and a K-projective complex respectively. Moreover, any quasi-isomorphism between K-injective (resp. K-projective) complexes is an isomorphism in $\cK(RC)$.
\begin{proof}
    See, for example, \cite[Propositions 4.3.1, 4.3.4]{Kra} and their duals; note that they both apply by \cref{GroCat}.
\end{proof}
\end{lem}

\begin{rem}[{\cite[Remark 4.3.5]{Kra}}]\label{KInj}
    The subcategory of K-injective (resp. K-projective) complexes is the smallest triangulated subcategory of $\cK(RC)$ closed under products (resp. coproducts) and containing stalk complexes of injective (resp. projective) objects. In particular, $\mathbf{i}X$ and $\mathbf{p}X$ can always be chosen degreewise injective and degreewise projective respectively.
\end{rem}

\begin{prop}\label{FunProp}
    For any small category $C$ and rings $R$ and $S$, the following holds:
    \begin{enumerate}
        \item Any additive functor $F:\Mods(R)\to\Mods(S)$
        \begin{enumerate}
            \item Extends to an additive functor $\bar{F}:\Mods_R(RC)\to\Mods_S(SC)$;
            \item Admits a right and a left derived functor $\mathbf{R}\bar{F},\mathbf{L}\bar{F}:\cD(RC)\to\cD(SC)$.
        \end{enumerate}
        \item Any adjoint pair $F\dashv G:\Mods(R)\to\Mods(S)$
        \begin{enumerate}
            \item Extends to an adjoint pair $\bar{F}\dashv\bar{G}:\Mods_R(RC)\to\Mods_S(SC)$;
            \item Extends to an adjoint pair $\mathbf{L}\bar{F}\dashv\mathbf{R}\bar{G}:\cD(RC)\to\cD(SC)$.
        \end{enumerate}
    \end{enumerate}
\begin{proof}
    (1) For any $X,X'\in\Mods_R(RC)$ and any natural transformation $f:X\to X'$, define $\bar{F}X:=F\circ X$ and $\bar{F}f:\bar{F}X\to\bar{F}X'$ such that $(\bar{F}f)_c:=F(f_c)$ for any $c\in C$. Then (b) follows by \cref{Res} and \cite[Theorem 8.4.9, Theorem 8.4.23]{Yek}.
    
    \noindent (2) Let $\eta:1_{\Mods(R)}\to GF$ and $\varepsilon:FG\to 1_{\Mods(S)}$ be the unit and the counit of the adjunction, i.e. natural transformations satisfying, for any $M\in\Mods(R)$ and $N\in\Mods(S)$, the triangle identities
    \[\varepsilon_{F(M)}\circ F(\eta_M)=\id_{F(M)} \text{ and } G(\varepsilon_N)\circ\eta_{G(N)}=\id_{G(N)}\]  
    Define the natural transformations $\bar{\eta}:1_{\Mods_R(RC)}\to \bar{G}\bar{F}$ and $\bar{\varepsilon}:\bar{F}\bar{G}\to 1_{\Mods_S(SC)}$ such that, for any $X\in\Mods_R(RC)$ and $Y\in\Mods_S(SC)$, ${(\bar{\eta}_X)}_c:=\eta_{X(c)}$ and ${(\bar{\varepsilon}_Y)}_c:=\varepsilon_{Y(c)}$ for any $c\in C$. It is straightforward to check that $\bar{\eta}$ and $\bar{\varepsilon}$ are well defined, i.e. for any $\alpha:c\to d$  the diagrams
    \[\begin{tikzcd} X(c)\arrow[r,"{(\bar{\eta}_X)}_c"]\arrow[d,"X(\alpha)"] & \bar{G}\bar{F}X(c)\arrow[d,"\bar{G}\bar{F}X(\alpha)"] \\ X(d)\arrow[r,"{(\bar{\eta}_X)}_d"] & \bar{G}\bar{F}X(d) \end{tikzcd} \text{ and } \begin{tikzcd} \bar{G}\bar{F}Y(c)\arrow[r,"{(\bar{\varepsilon}_Y)}_c"]\arrow[d,"\bar{G}\bar{F}Y(\alpha)"] & Y(c)\arrow[d,"Y(\alpha)"] \\ \bar{G}\bar{F}Y(d)\arrow[r,"{(\bar{\varepsilon}_Y)}_d"] & Y(d) \end{tikzcd}\]
    commute and for any $f:X\to X'$ and $g:Y\to Y'$ the diagrams
    \[\begin{tikzcd} X\arrow[r,"\bar{\eta}_X"]\arrow[d,"f"] & \bar{G}\bar{F}X\arrow[d,"\bar{G}\bar{F}f"]\\X'\arrow[r,"\bar{\eta}_{X'}"] & \bar{G}\bar{F}X' \end{tikzcd} \text{ and } \begin{tikzcd} \bar{F}\bar{G}Y\arrow[r,"\bar{\varepsilon}_Y"]\arrow[d,"\bar{F}\bar{G}g"] & Y\arrow[d,"g"] \\ \bar{F}\bar{G}Y'\arrow[r,"\bar{\varepsilon}_{Y'}"] & Y' \end{tikzcd}\]
    commute, and satisfy the triangle identities
    \[\bar{\varepsilon}_{\bar{F}X}\circ\bar{F}\bar{\eta}_X=\id_{\bar{F}(X)} \text{ and } \bar{G}\bar{\varepsilon}_Y\circ\bar{\eta}_{\bar{G}Y}=\id_{\bar{G}(Y)}\]
    Thus, $\bar{\eta}$ and $\bar{\varepsilon}$ are the unit and the counit of the adjunction $\bar{F}\dashv\bar{G}$. Then (b) follows by \cite[Theorem 1.7.1]{Mil}.
\end{proof}
\end{prop}

When $R$ is a commutative ring, as it will be from now on, there are two actions of the category $\Mods(R)$ on $\Mods_R(RC)$ which plays a relevant role. These are given by the bifunctors
\[\_\otimes_R\_,\Hom_R(\_,\_):\Mods(R)\times\Mods_R(RC)\to\Mods_R(RC)\]
where for every $M\in\Mods(R)$, $X\in\Mods_R(RC)$ and $c\in C$, they are defined as
\[(M\otimes_R X)(c):=M\otimes_R X(c) \text{ and } \Hom_R(M,X)(c):=\Hom_R(M,X(c))\]

\begin{rem}
    These actions extend to the homotopy categories in the standard way, i.e. by taking, for any $R$-complex $M$ and $RC$-complex $X$, the $\amalg$-totalization of the double complex $M^i\otimes_R X^j$ and the $\Pi$-totalization of the double complex $\Hom_R(M^{-i},X^j)$.
\end{rem}

Note that, K-projective complexes of $\cK(R)$ are exactly those $P$ which make the functor $\Hom_R(P,\_):\cK(RC)\to\cK(RC)$ preserve acyclicity. Analogously, we call a complex $F\in\cK(R)$ \textit{K-flat} if the functor $F\otimes_R\_:\cK(RC)\to\cK(RC)$ preserves acyclicity. Moreover, we say that a complex $X\in\cK(RC)$ is \textit{$R$-K-flat} (resp. \textit{$R$-K-injective}) if the functor $\_\otimes_R X:\cK(R)\to\cK(RC)$ (resp. $\Hom_R(\_,X):\cK(R)\to\cK(RC)$) sends acyclic $R$-complexes to acyclic $RC$-complexes. Moreover, we call a functor $X\in\Mods_R(RC)$ \textit{objectwise} flat (resp. injective) if $X(c)$ is a flat (resp. injective) $R$-module for any $c\in C$.

The next lemma shows some examples of these complexes and the relationship between them.

\begin{lem}\label{SmartRes}
    The following holds:
    \begin{enumerate}
        \item K-projective complexes of $\cK(R)$ and bounded above complexes of flat modules are K-flat;
        \item K-projective complexes of $\cK(RC)$ and bounded above complexes of objectwise flat functors are $R$-K-flat;
        \item K-injective complexes of $\cK(RC)$ and bounded below complexes of objectwise injective functors are $R$-K-injective.
    \end{enumerate}  
\begin{proof}
    Let us prove (3), the others follow similarly. Consider the class
    \[\leftcurly Y\in\cK(RC)\mid\Hom_R(\_,Y) \text{ preserves acyclicity}\rightcurly\] 
    It is a triangulated subcategory of $\cK(RC)$ closed under products and containing stalk complexes of objectwise injective functors. In particular, it contains stalk complexes of injective functors and so, by \cref{KInj}, it contains K-injective complexes of $\cK(RC)$. Moreover, since bounded below complexes of objectwise injective functors can be reconstructed, via extensions and products, from bounded ones, they are contained in the above class too.
\end{proof}
\end{lem}

\begin{thm}\label{DerAct}
    We have that:
    \begin{enumerate}
        \item\begin{enumerate}
            \item The bifunctor $\_\otimes_R\_$ admits a left-derived functor $\_\otimesL\_$, which can be computed, independently up to isomorphism, either using K-flat resolutions in the first variable or $R$-K-flat resolutions in the second;
            \item The bifunctor $\Hom_R(\_,\_)$ admits a right-derived functor $\RHom_R(\_,\_)$, which can be computed, independently up to isomorphism, either using K-projective resolutions in the first variable or $R$-K-injective resolutions in the second;
        \end{enumerate}
        \item For any $R$-complex $M$ the functors $M\otimesL\_$ and $\RHom_R(M,\_)$ form an adjoint pair. In particular, for any two $RC$-complexes $X$ and $Y$, we have the isomorphism
        \[\Hom_{\cD(RC)}(M\otimesL X,Y)\cong\Hom_{\cD(RC)}(X,\RHom_R(M,Y))\]
    \end{enumerate}
\begin{proof}
    (1) By \cref{Res} and \cref{SmartRes}, $\cK(R)$ admits K-flat (resp. K-projective) resolutions and $\cK(RC)$ admits $R$-K-flat (resp. $R$-K-injective) resolutions and, for any quasi-isomorphisms $f:M\to M'$ between K-flat (resp. K-projective) complexes and $g:X\to X'$ between $R$-K-flat (resp. $R$-K-injective) complexes, the morphism $f\otimes_R g$ (resp $\Hom_R(f,g)$) is a quasi-isomorphism. Thus, the statement follows from \cite[Theorem 9.3.16]{Yek} (resp. \cite[Theorem 9.3.11, Remark 9.3.18]{Yek}). Moreover, denoting by $\mathbf{f}M\to M$ and $\mathbf{f_R}X\to X$ a K-flat resolution of $M$ and a $R$-K-flat resolution of $X$ (resp. $\mathbf{p}M\to M$ and $X\to\mathbf{i_R}X$ for K-projective and $R$-K-injective), we have that
    \[(M\otimesL\_)(X):=M\otimes_R\mathbf{f_R}X\cong\mathbf{f}M\otimes_R\mathbf{f_R}X\cong\mathbf{f}M\otimes_R X=:(\_\otimesL X)(M)\]
    (resp. $\RHom_R(M,\_)(X)\cong\Hom_R(\mathbf{p}M,\mathbf{i_R}X)\cong\RHom_R(\_,X)(M)$).
    
    \noindent (2) Follows from \cref{FunProp}.
\end{proof}
\end{thm}

\subsection{t-structures}\label{t-str}
Let $\cT=\cD(\cA)$ be the derived category of a Grothendieck category and assume it is compactly generated (as it will be in the settings relevant to us, see \cref{DerCat}). A subcategory of $\cT$ is called \textit{suspended} (resp. \textit{cosuspended}) if it is closed under direct summands and positive (resp. negative) shifts and extensions; it is called \textit{cocomplete} (resp. \textit{complete}) if it is closed under coproducts (resp. products) and \textit{homotopically smashing} if it is closed under directed homotopy colimits; where, for a small category $I$, the \textit{homotopy colimit} functor $\mathrm{hocolim}_I:\cD\left(\cA^I\right)\to\cD(\cA)$ is meant to be the left derived functor of the colimit $\colim_I$ (see \cite[Example A.2]{HN} for a nice treatment in this context). Given a collection of objects $\mathcal{S}\subseteq\cT$ we will denote by $\susp_\cT\leftangle\mathcal{S}\rightangle$ (resp. $\cosusp_\cT\leftangle\mathcal{S}\rightangle$) the smallest suspended (resp. cosuspended) subcategory of $\cT$ containing the object in $\mathcal{S}$. We will use the decorations $\amalg,\Pi,\hs$ in the apex to mean that the subcategory is the smallest among the cocomplete, complete and homotopically smashing, respectively.

A \textit{t-structure} in $\cT$ consists of a pair of subcategories $(\cU,\cV)$ satisfying:
\begin{itemize}
    \item[(i)] $\Hom_\cT(\cU,\cV)=0$;
    \item[(ii)] For any $X\in\cT$, there is a distinguished triangle
    \[\begin{tikzcd} U \arrow[r] & X \arrow[r] & V \arrow[r,"+"] & {} \end{tikzcd}\]
    with $U\in\cU$ and $V\in\cV$;
    \item[(iii)] $\cU[1]\subseteq\cU$ (equivalently $\cV[-1]\subseteq\cV$).
\end{itemize}
In this case it holds that $\cU={}^\perp\cV$ and $\cV=\cU^\perp$. We call $\cU$ and $\cV$ the \textit{aisle} and the \textit{coaisle} of the t-structure, respectively. Recall that any aisle is a cocomplete suspended subcategory closed under homotopy colimits (see \cite[Proposition 5.2]{SSV}) and, by \cite{KV}, a suspended subcategory is an aisle if and only if the inclusion functor has a right adjoint, which we call the \textit{truncation functor} of the t-structure; the dual holds for coaisles and cosuspended subcategories.

\begin{rem}\label{GenRem}
    By \cite[Theorem 3.4]{AJSo}, for any set of objects $\cS\subseteq\cT$, the subcategory $\susp_\cT^\amalg(\cS)$ is equal to the aisle ${}^\perp(\cS[\geq\!0]^\perp)$, in the following denoted by $\aisle_\cT\leftangle\mathcal{S}\rightangle$. The analogue of this last statement for coaisles can be achieved in the following, more restricted, context. If $\cT=\cD(\bK Q)$ is the derived category of an hereditary algebra of finite representation type, by \cite[Theorem 12.20]{Bel}, it is a pure-semisimple triangulated category, i.e. any object is pure-injective. By \cite[Proposition 5.10, Corollary 5.11]{LV}, for any set of objects $\mathcal{E}\subseteq\cD(\bK Q)$, the subcategory $\cosusp_{\bK Q}^\Pi(\mathcal{E})$, automatically closed under pure-monomorphisms, is equal to the coaisle $({}^\perp\mathcal{E}[\leq\!0])^\perp$ (we refer to \textit{loc. cit.} for the undefined terminology). Moreover, since it is right orthogonal to a compactly generated aisle, it is homotopically smashing, and so equal to $\cosusp_{\bK Q}^{\Pi,\hs}(\mathcal{E})$. We will use this fact later in \cref{CoaisleAss}.
\end{rem}

A t-structure $(\cU,\cV)$ in $\cT$ is called:
\begin{itemize}[label=$\boldsymbol{\cdot}$]
    \item \textit{Compactly generated} if $\cU=\aisle_\cT\leftangle\cS\rightangle$, for some $\cS\subseteq\cT^c$;
    \item \textit{Homotopically smashing} if $\cV$ is closed under directed homotopy colimit;
    \item \textit{Stable} if $\cU$ and $\cV$ are triangulated subcategories.
\end{itemize}
Note that localizing subcategories, in the sense of Bousfield, are precisely the aisles of stable t-structures and, by \cite[Theorem A]{K}, a stable t-structure $(\cU,\cV)$ is homotopically smashing if and only if $\cU$ is a smashing subcategory. By \cite[Proposition 7.2]{SSV}, any compactly generated t-structure is homotopically smashing, and we say that the (\textit{stable}) \textit{telescope conjecture} holds in $\cT$ if any (stable) homotopically smashing t-structure is compactly generated. In particular, \cref{Main} (2) will answer affirmatively to this question in the case of t-structures of $\cD(RQ)$, where $R$ is a commutative noetherian ring and $Q$ a Dynkin quiver.

The next lemma shows how the t-structures in $\cD(RC)$ are affected by the actions of $\cD(R)$ described before.

\begin{lem}\label{ClosureLem}
    For any object $X\in\cD(RC)$ the following holds:
    \begin{enumerate}
        \item The class $\leftcurly M\in\cD(R)\mid M\otimesL X\in\susp_{RC}^\amalg\leftangle X\rightangle\rightcurly$ is a cocomplete suspended subcategory of $\cD(R)$ containing $\cD^{\leq0}(R)$;
        \item The class $\leftcurly M\in\cD(R)\mid\RHom_R(M,X)\in\cosusp_{RC}^\Pi\leftangle X\rightangle\rightcurly$ is a cocomplete suspended subcategory of $\cD(R)$ containing $\cD^{\leq0}(R)$;
        \item The class $\leftcurly M\in\cD(R)\mid M\otimesL X\in\cosusp_{RC}^\hs\leftangle X\rightangle\rightcurly$ is a homotopically smashing cosuspended subcategory of $\cD(R)$ containing $\cD^{[0,n]}(\Flat(R))$ for any $n\geq0$;
        \item If $X$ is compact, the class $\leftcurly M\in\cD(R)\mid M\otimesL X\in\cosusp_{RC}^{\Pi,\hs}\leftangle X\rightangle\rightcurly$ is a homotopically smashing complete cosuspended subcategory of $\cD(R)$.
    \end{enumerate}
\begin{proof}
    (1-2) The first two classes are trivially closed under positive shifts, extensions, coproducts and contain the stalk complex $R[0]$. Thus, by \cref{GenRem}, they contain the standard aisle $\cD^{\leq0}(R)$.
    
    \noindent (3) It is also clear that the third class is closed under negative shifts, extensions and directed homotopy colimits and contains the stalk complex $R[0]$. Thus, by Lazard's theorem (see \cite[5.5.7]{CFH}), it contains the stalk complex $F[0]$ for any flat $R$-module $F$. By closure under negative shifts and extensions, it contains all bounded complexes of flat modules concentrated in nonnegative degrees.
    
    \noindent (4) If $X$ is compact, it is isomorphic to a bounded complex of objectwise finitely presented functors and so the functor $\_\otimesL X$ commutes with products (see \cite[3.1.31]{CFH}).
\end{proof}
\end{lem}

\begin{prop}\label{ClosureProp}
    For any t-structure $(\cU,\cV)$ in $\cD(RC)$ the following holds:
    \begin{enumerate}
        \item $\cD^{\leq0}(R)\otimesL\cU\subseteq\cU$;
        \item $\RHom_R(\cD^{\leq0}(R),\cV)\subseteq\cV$;
        \item If $\cV$ is homotopically smashing, $\cD^{[0,n]}(\Flat(R))\otimesL\cV\subseteq\cV$ for any $n\geq0$.
    \end{enumerate}
\end{prop}
\begin{proof}
    It follows by \cref{ClosureLem} (1), (2), (3) respectively.
\end{proof}

\subsection{Support theory}
Let us now introduce some notions of support theory in $\cD(RC)$. Our main references for this section are \cite{S,Stour} and, from now on, we assume the ring $R$ to be commutative and noetherian.

Recall that a subset $V$ of $\Spec(R)$ is called \textit{specialization closed} if for any $\p\in V$ and $\q\in\Spec(R)$ such that $\p\subseteq\q$ it holds that $\q\in V$. By \cite[Theorem 3.3]{NeeTC}, for any specialization closed subset $V\subseteq\Spec(R)$, there is a (smashing) stable t-structure $(\cU_V,\cU_V^\perp)$ in $\cD(R)$, where $\cU_V=\{M\in\cD(R)\mid\supp(M)\subseteq V\}$ and $\supp(M)=\{\p\in\Spec(R)\mid\kp\otimesL M\neq0\}$. Denoting by $\Gamma_V$ and $\mathrm{L}_V$ the truncation functors, by abuse of notation, we define in $\cD(RC)$ the functors
\[\Gamma_V:=\Gamma_V R\otimesL\_ \text{ and } \mathrm{L}_V:=\mathrm{L}_V R\otimesL\_\]
By \cite[Lemma 4.3, 4.4, 4.11]{S},
for a specialization closed subset $V\subseteq\Spec(R)$, the essential images of $\Gamma_V$ and $\mathrm{L}_V$ form a compactly generated stable t-structure $(\Gamma_V\cD(RC),\mathrm{L}_V\cD(RC))$. In particular, for any complex $X\in\cD(RC)$, we can find a distinguished triangle
\[\begin{tikzcd} \Gamma_V X \arrow[r,"\gamma_X"] & X \arrow[r,"\lambda_X"] & \mathrm{L}_V X \arrow[r,"+"] & {} \end{tikzcd}\]
and the functors $\Gamma_V$ and $\mathrm{L}_V$ satisfy
\[\im\Gamma_V=\ker\mathrm{L}_V={}^\perp(\im\mathrm{L}_V) \text{ and } \im\mathrm{L}_V=\ker\Gamma_V=(\im\Gamma_V)^\perp\]
Given a prime ideal $\p\in\Spec(R)$, consider the specialization closed subsets
\[V(\p)=\leftcurly\q\in\Spec(R)\mid\p\subseteq\q\rightcurly \text { and }Z(\p)=\leftcurly\q\in\Spec(R)\mid\q\nsubseteq\p\rightcurly\]
then we define the functor
\[\Gamma_\p:=\mathrm{L}_{Z(\p)}\circ\Gamma_{V(\p)}\cong\Gamma_{V(\p)}\circ\mathrm{L}_{Z(\p)}\]
We will refer to its essential image $\Gamma_\p\cD(RC)$ as the \textit{stalk subcategory} of $\cD(RC)$ at $\p$ and, for a complex $X\in\cD(RC)$, we define the support of $X$ over $R$ as
\[\supp_R(X)=\leftcurly\p\in\Spec(R)\mid\Gamma_\p X\neq0\rightcurly\]

\begin{rem}
    This notion of support agrees with the usual one of homological algebra. Indeed, by \cite[Lemma 3.22]{Stour}, for any complex $X\in\cD(RC)$ we have that $\supp_R(X)=\leftcurly\p\in\Spec(R)\mid\kp\otimesL X\neq0\rightcurly$. However, the definition given here turned out to be more efficient because of the nice properties of the functor $\Gamma_\p$ instead of those of $\kp\otimesL\_$, such as idempotence and \cref{SuppProp} (3). 
\end{rem}

\begin{prop}\label{SuppProp}
    For any specialization closed subset $V$ of $\Spec(R)$ and any complex $X\in\cD(RC)$ the following holds:
    \begin{enumerate}
        \item $\supp_R(X)\subseteq V$ if and only if $\mathrm{L}_V X=0$ if and only if $\gamma_X:\Gamma_V X\to X$ is an isomorphism;
        \item $\supp_R(X)\cap V=\emptyset$ if and only if $\Gamma_V X=0$ if and only if $\lambda_X:X\to\mathrm{L}_V X$ is an isomorphism;
        \item $\supp_R(X)\subseteq\{\p\}$ if and only if $X\in\Gamma_\p\cD(RC)$ if and only if $\Gamma_\p X\cong X$.
    \end{enumerate}
\begin{proof}
    (1) By \cite[Proposition 5.7 (4)]{S}, the support of $X$ is contained in $V$ if and only if the support of $\mathrm{L}_V X$ is empty thus, by \cite[Theorem 6.9 (ii)]{S}, if and only if $\mathrm{L}_V X=0$. Since $\ker\mathrm{L}_V=\im\Gamma_V$, this is the case if and only if $X\cong\Gamma_V X'$ for some $X'\in\cD(RC)$ and in this case, by idempotence of the functor $\Gamma_V$, we have that $\Gamma_V X\cong\Gamma_V^2 X'\cong\Gamma_V X'\cong X$.
    
    \noindent (2) Similar to previous point.
    
    \noindent (3) The support of $X$ is contained in $\{\p\}$ if and only if $\Gamma_\q X=0$ for any $\q\neq\p$. In this case, by \cite[Theorem 6.9 (i)]{S}, $X$ lies in the localizing subcategory generated by $\Gamma_\p X$ and thus in $\Gamma_\p\cD(RC)$.  On the other hand, if  $X\cong\Gamma_\p X'$ for some $X'\in\cD(RC)$, by \cite[Corollary 5.8]{S}, the support of $X$ is contained in $\{\p\}$ and, by idempotence of the functor $\Gamma_\p$, we have that $\Gamma_\p X\cong\Gamma_\p^2 X'\cong\Gamma_\p X'\cong X$.
\end{proof}
\end{prop}

For the specialization closed subsets $V(\p)$ and $Z(\p)$, we can find an explicit description of the functors $\mathrm{L}_{Z(\p)}$ and $\Gamma_{V(\p)}$ (see \cite[Example 2.24]{Stour}) through the $\otimes$-idempotent objects
\[\mathrm{L}_{Z(\p)}R\cong R_\p\text{ and }\Gamma_{V(\p)}R\cong K_\infty(\p)\]
where the complex $K_\infty(\p)$ is the \textit{infinite Koszul complex} at $\p$ (also known as stable Koszul complex or {\v C}ech complex and denoted by $\check{C}(\p)$, as in \cite{CFH}).

\begin{defn}
    For an ideal $\a$ in $R$, let $n$ be the minimal number of generators of $\a$ and $\{a_1,\ldots,a_n\}$ be a set of generators. The \textit{Koszul complex} (on $R$) at $\a$ is defined to be
    \[K(\a)=\bigotimes_{i=1}^n R\xrightarrow{a_i\cdot}R\]
    where the complex $R\xrightarrow{a_i\cdot}R$ is concentrated in degrees $-1$ and $0$ and the map is multiplication by $a_i$. It follows from \cite[14.4.29]{CFH} that, in some situations, the definition may be independent of the choice of generators. However, in general this is not the case, so for each finitely generated ideal $\a$ we have to chose a set of generators of minimal cardinality.
    
    The \textit{infinite Koszul complex} (on $R$) at $\a$ is defined to be
    \[K_\infty(\a)=\bigotimes_{i=1}^n R\xrightarrow{\rho_i}R[a_i^{-1}]\]
    where the complex $R\xrightarrow{\rho_i}R[a_i^{-1}]$ is concentrated in degrees $0$ and $1$ and the map $\rho_i$ is given by $r\mapsto\frac{r}{1}$. It follows from \cite[11.4.16]{CFH} that the definition is independent, up to quasi-isomorphism, of the choice of the set of generators.
\end{defn}

Note that $K(\a)$ is a complex of finitely generated free $R$-modules concentrated in degrees $[-n, 0]$, thus it is compact. In particular, $K(\a)^{-i}=R^{\binom{n}{i}}$ for any $0\leq i\leq n$ and the differentials are represented by matrices with entries in $\a$. Instead, $K_\infty(\a)$ is a bounded complex of flat $R$-modules concentrated in degrees $[0, n]$.

\begin{rem}\label{Koszul}
    We want to stress out some properties of these complexes:
    \begin{itemize}
        \item For any prime ideal $\p$, it holds
        \[\kp\otimesL K(\a)=
        \begin{cases}
            \bigoplus\limits_{i=0}^n\kp^{\binom{n}{i}}[i] &\text{if }\a\subseteq\p\\
            \quad\quad\quad 0 &\text{otherwise}
        \end{cases}\]
        Indeed, by \cite[15.1.10]{CFH} the complex is nonzero if and only if $\a\subseteq\p$ and in this case, tensoring by $\kp$ annihilates multiplications by elements of $\a$, so it annihilates all the differentials of $K(\a)$.
        \item \cite[14.3.2]{CFH} Known as \textit{self-duality} property: there is an isomorphism of functors
        \[\RHom_R(K(\a),\_)\cong K(\a)[-n]\otimesL\_\]
        \item \cite[11.4.12]{CFH} Let $\a^{(t)}$ denote the ideal $\left(a_1^t,\ldots,a_n^t\right)$, then
        \[K_\infty(\a)\cong\hocolim_{t\in\bN}K\left(\a^{(t)}\right)[-n]\]
    \end{itemize}
\end{rem}

\begin{rem}\label{GammaV}
    Actually, by \cite[Theorem 1.6]{HTC}, we can find an explicit description of the functor $\Gamma_V$ for any specialization closed subset $V$. Indeed, writing $V=\bigcup_{i\in I}V(\p_i)$ as an union of Zariski closed subsets of $\Spec(R)$ and denoting by $\mathcal{F}$ the lattice of finite subsets of $I$ and, for any $F\in\mathcal{F}$, by $V_F=\bigcup_{i\in F}V(\p_i)$, in \textit{loc. cit.} it is proved the isomorphism of triangles
    \[\begin{tikzcd} \Gamma_V R \arrow[r] \arrow[d,"\|\wr", phantom] & R \arrow[r] \arrow[d,"\|", phantom] & \mathrm{L}_V R \arrow[d,"\|\wr", phantom] \arrow[r,"+"] & {} \\
    \hocolim_{F\in\mathcal{F}}\Gamma_{V_F}R \arrow[r] & R \arrow[r] & \hocolim_{F\in\mathcal{F}}\mathrm{L}_{V_F}R \arrow[r,"+"] & {} \end{tikzcd}\]
    Note that $V_F$ is equal to the Zariski closed $V(\a_F)$ corresponding to the ideal $\a_F=\prod\limits_{i\in F}\p_i$. Thus, by \cite[Example 2.24]{Stour}, it holds that $\Gamma_{V_F}R\cong K_\infty(\a_F)$ and we get that
    \[\Gamma_V\cong\hocolim_{F\in\mathcal{F}}K_\infty(\a_F)\otimesL\_\cong\hocolim_{(F,t)\in\mathcal{F}\times\bN}\RHom_R\left(K\left(\a_F^{(t)}\right),\_\right)\] 
\end{rem}

\section{Local to global principle and Minimality}\label{LTGM}
The aim of this section is to generalize two properties of localizing subcategories of $\cD(RC)$ -- the \textit{local to global principle} and the \textit{minimality of stalk subcategories} -- to the setting of homotopically smashing cosuspended subcategories. These properties will allow us to reduce the problem of generating homotopically smashing coaisles, first from $\cD(RC)$ to the stalk subcategories $\Gamma_\p\cD(RC)$, and then to the better understood derived categories $\cD(\kp C)$. First introduced in \cite[Section 4.2]{BIK}, these property characterize, among the compactly generated $R$-linear triangulated categories (in the sense of \cite[Section 2.3]{BIK}), those which the authors called the \textit{stratified} ones. For a compactly generated $R$-linear category $\cT$, the authors develop a theory of support, equivalent to that of \cite{S} by \cite[Section 9]{S}, and identify the support of $\cT$ with a certain subset of $\Spec(R)$, denoted $\supp_R(\cT)$. In this setting, stratified $R$-linear categories are precisely those whose lattice of localizing subcategories, $\mathrm{Loc}(\cT)$, is parametrized by the subsets of $\supp_R(\cT)$; see \cite[Theorem 4.2]{BIK}. In particular, the local to global principle states that:
\begin{itemize}[itemindent=10pt]
    \item[\textsf{(LTG)}] For any $X\in\cT$, $\loc_\cT\leftangle X\rightangle=\loc_\cT\leftangle\Gamma_\p X\mid\p\in\Spec(R)\rightangle$;
\end{itemize}
and, by \cite[Proposition 3.6]{BIK}, it characterizes the categories for which $\mathrm{Loc}(\cT)$ is in bijection with the tuples $(\cS(\p))_{\p\in\Supp_R(\cT)}$, where $\cS(\p)\in\mathrm{Loc}(\Gamma_\p\cT)$ for any $\p\in\Spec(R)$. However, in the context of representations of small categories, stratification is not the right notion to consider; indeed, while the local to global principle holds in $\cD(RC)$, for any commutative noetherian ring $R$ and small category $C$ \cite[Theorem 3.5]{AS}, the minimality fails even in $\cD(\bK Q)$, for a field $\bK$ and a Dynkin quiver $Q$ \cite[Example 4.6]{BIK}.

In order to obtain a classification of the localizing subcategories of $(\cD(RC))$ and then prove the stable telescope conjecture (see \cref{t-str}), in \cite[Theorem 4.2]{AS} the authors replaced the minimality condition, with the following equivalent conditions:
\begin{itemize}
    \item[\textsf{(M)}] For any $\p\in\Spec(R)$ and $X\in\Gamma_\p\cD(RC)$, $\loc_{RC}\leftangle X\rightangle=\loc_{RC}\leftangle \kp\otimesL X\rightangle$;
    \item[\textsf{(M')}\!] For any $\p\in\Spec(R)$, there is a bijection between $\mathrm{Loc}(\Gamma_\p\cD(RC))$ and $\mathrm{Loc}(\cD(\kp C))$.
\end{itemize}
Even if these conditions are not equivalent to the minimality in \cite{BIK} for a general small category $C$, they are in the case of $\cD(R)$ and so, in a sense, they transfer the notion of minimality from the stratified categories to our context; so, it still makes sense to call them minimality conditions in $\cD(RC)$.

\subsection{Local to global principle}
In the following we will generalize \textsf{(LTG)} from localizing subcategories to cosuspended subcategories closed under homotopy colimits, thus avoiding the closure under positive shifts. Note that, while in the case of stable t-structures \textsf{(LTG)} is used to study localizing subcategories, i.e. the left side of the pair, our generalization also applies to any homotopically smashing t-structures, but on the side of the coaisles.

Let us start defining inductively a transfinite collection $\{W_\alpha\}$ of specialization closed subsets of $\Spec(R)$. Denoting by $\max W$ the set of maximal elements of $W$, let:
\begin{itemize}[label=$\boldsymbol{\cdot}$]
    \item $W_0=\max\Spec(R)$;
    \item $W_{\alpha+1}=W_\alpha\cup\max(\Spec(R)\setminus W_\alpha)$, for any successor ordinal $\alpha+1$;
    \item $W_\lambda=\bigcup\limits_{\beta<\lambda}W_\beta$, for any limit ordinal $\lambda$.
\end{itemize}
In this way, at some point, we reach the least ordinal $\delta$ satisfying $W_\delta=\Spec(R)$.

\begin{lem}\label{TransInd}
    Let $Y\in\cD(RC)$, then:
    \begin{enumerate}
        \item For any successor ordinal $\alpha+1$, there exists a distinguished triangle
        \[\begin{tikzcd} \Gamma_{W_\alpha}Y \arrow[r] & \Gamma_{W_{\alpha+1}}Y \arrow[r] & \bigoplus\limits_{\p\in W_{\alpha+1}\setminus W_\alpha}\Gamma_\p Y \arrow[r,"+"] & {} \end{tikzcd}\]
        \item  For any limit ordinal $\lambda$, we have that $\Gamma_{W_\lambda}Y\cong\hocolim_{\beta<\lambda}\Gamma_{W_\beta}Y$.
    \end{enumerate}
\begin{proof}
    (1) By \cite[Remark 4.2]{HNS}, for each ordinal $\alpha+1$, the following is a distinguished triangle in $\cD(R)$ 
    \[\begin{tikzcd} \Gamma_{W_\alpha}R \arrow[r] & \Gamma_{W_{\alpha+1}}R \arrow[r] & \bigoplus\limits_{\p\in W_{\alpha+1}\setminus W_\alpha}\Gamma_\p R \arrow[r,"+"] & {} \end{tikzcd}\]
    (beware of the change of notation in the reference: the functors $\mathbf{R}\Gamma_V$ and $\mathbf{R}\Gamma_\p$ there stand for $\Gamma_V$ and $\Gamma_{V(\p)}$ here). Then, applying $\_\otimesL Y$ to it, we get the triangle in the statement.
    
    \noindent (2) By \cite[Lemma 6.6]{S}, we have that $\Gamma_{W_\lambda}R\cong\hocolim_{\beta<\lambda}\Gamma_{W_\beta}R$, then the result follows applying $\_\otimesL Y$ to both sides.
\end{proof}
\end{lem}

\begin{thm}[\bf Local to Global Principle]\label{LTG}
    For any complex $Y\in\cD(RC)$, we have that
    \[\cosusp_{RC}^\hs\leftangle Y\rightangle=\cosusp_{RC}^\hs\leftangle\Gamma_\p Y\mid\p\in\Spec(R)\rightangle\]
\begin{proof}
    By \cref{TransInd}, any complex $Y$ is isomorphic to $\Gamma_{W_\delta}Y$, where $\delta$ is the least ordinal such that $W_\delta=\Spec(R)$. In particular, $Y$ can be reconstructed, through extensions, coproducts and direct homotopy colimits, from the complexes $\{\Gamma_\p Y\}_{\p\in\Spec(R)}$. Thus, the containment from left to right follows. On the other hand, $\Gamma_\p Y:=(K_\infty(\p)\otimesL R_\p)\otimesL Y$ is given by tensoring $Y$ with a bounded complex of flat modules concentrated in nonnegative degrees and so, by \cref{ClosureLem} (3), it is contained in $\cosusp_{RC}^\hs\leftangle Y\rightangle$ for any prime ideal $\p$.
\end{proof}
\end{thm}

Note that the last containment is the one that forces us to switch to coaisles, in fact aisles of non stable t-structure are never closed under tensor products with complexes concentrated in nonnegative degrees.

\subsection{Minimality of stalk subcategories}
The following theorem is the analogue of \textsf{(M)} for cosuspended subcategories closed under products and directed homotopy colimits. Unlike the theorem above, this is not a generalization of \textsf{(M)}, in fact it does not apply to localizing subcategories, since they are not closed under products in general. However, we will prove in an upcoming work that this result implies \textsf{(M')} for homotopically smashing coaisles in the context of quiver representations, i.e. for any $\p\in\Spec(R)$, there is a bijection between $\mathrm{Coaisle_{hs}}(\Gamma_\p\cD(RQ))$ and $\mathrm{Coaisle_{hs}}(\cD(\kp Q))$ for any finite quiver $Q$ (see \cite{HS}).

\begin{thm}[\bf Minimality of Stalk Subcategories]\label{Min}
    For any complex $Y\in\Gamma_\p\cD(RC)$, we have that
    \[\cosusp_{RC}^{\Pi,\hs}\leftangle Y\rightangle=\cosusp_{RC}^{\Pi,\hs}\leftangle\RHom_R(\kp, Y)\rightangle\]
\begin{proof}
    The containment from right to left follows from \cref{ClosureLem} (2). For the other direction, consider the class
    \[\mathcal{M}=\leftcurly M\in\cD(R_\p)\mid\RHom_R(M, Y)\in\cosusp_{RC}^{\Pi,\hs}\leftangle\RHom_R(\kp, Y)\rightangle\rightcurly\]
    It is a cocomplete suspended subcategory of $\cD(R_\p)$ containing $\kp\cong R_\p/\p_\p$. Thus, by \cref{GenRem} and \cite[Proposition
    2.4]{AJS}, it contains $R_\p/\p_\p^{(t)}$ for any $t\geq1$ and, by \cite[Lemma 5.3]{Hrb}, it contains the Koszul complexes $K^{R_\p}\left(\p_\p^{(t)}\right)\cong K\left(\p^{(t)}\right)_\p$, where the first complex is computed over $R_p$ and the second over $R$.
    
    \noindent Consider now the class
    \[\leftcurly M\in\cD(R_\p)\mid M\otimesL Y\in\cosusp_{RC}^{\Pi,\hs}\leftangle\RHom_R(\kp, Y)\rightangle\rightcurly\]
    It is a homotopically smashing cosuspended subcategory of $\cD(R_\p)$. By self-duality of Koszul complexes in \cref{Koszul}, $\RHom_R(K\left(\p^{(t)}\right)_\p, Y)\cong K\left(\p^{(t)}\right)_\p[-n]\otimesL Y$, thus it contains the complex $K\left(\p^{(t)}\right)_\p[-n]$ for any $t\geq1$. Since it is closed under directed homotopy colimits, by \cref{Koszul}, it contains the the complex $K_\infty(\p)_\p=\Gamma_\p R$ and we can conclude that $Y\cong\Gamma_\p R\otimesL Y\in\cosusp_{RC}^{\Pi,\hs}\leftangle\RHom_R(\kp, Y)\rightangle$.
\end{proof}
\end{thm}

\begin{rem}
    Note that, we use products only in the first part of the proof to make $\mathcal{M}$ cocomplete and then apply \cref{GenRem}. It turns out that, with a bit more of work, we can prove that the inclusion $Y\in\cosusp_{RC}^\hs\leftangle\RHom_R(\kp, Y)\rightangle$ holds without closure under products. Indeed, in this case, the class
    \[\mathcal{M}=\leftcurly M\in\cD(R_\p)\mid\RHom_R(M, Y)\in\cosusp_{RC}^\hs\leftangle\RHom_R(\kp, Y)\rightangle\rightcurly\]
    is just a suspended subcategory of $\cD(R_\p)$ containing $\kp\cong R_\p/\p_\p$. For any $t\geq 1$, the Koszul complex $K^{R_\p}\left(\p_\p^{(t)}\right)$ is a bounded complex of finite free $R_\p$-modules concentrated in degrees $[-n,0]$ and, for any $i\in[-n,0]$, its cohomology module $H^i$ is a finitely generated $R_\p$-module annihilated by $\p_\p^{(t)}$ (see \cite[11.4.6 (a)]{CFH}). In particular, each $H^i$ admits a finite filtration with composition factors isomorphic to $\kp$ (see \cite[Lemma 10.62.1-2-4]{STA}). It follows that, for any $i$, the module $H^i$ is contained in $\mathcal{M}$ and, by closure under extensions, so does the complex $K^{R_\p}\left(\p_\p^{(t)}\right)\cong K\left(\p^{(t)}\right)_\p$. Then we can conclude as in the proof of \cref{Min}.
\end{rem}

\section{Dynkin quivers and Lattice lifts}\label{DynLift}
When the small category $C$ is the free category on a finite quiver $Q$, the category $\Mods_R(RC)$ is equivalent to the category of modules $\Mods(RQ)$, where $RQ$ is the free $R$-algebra on the set of paths of $Q$ with the composition of paths as product. In this case, the functors $X:C\to\Mods(R)$ can be seen as representations of the quiver $Q=(Q_0,Q_1)$ over the ring $R$, i.e. $X=(X_i,X_\alpha)_{i\in Q_0,\alpha\in Q_1}$ where $X_i$ is an $R$-module for any vertex $i\in Q_0$ and $X_\alpha:X_i\to X_j$ is an $R$-linear map for any arrow $\alpha:i\to j\in Q_1$.

Moreover, if $Q$ is a Dynkin quiver and $R=\bK$ a field, the path algebra $\bK Q$ is a representation-finite hereditary algebra and it is well-known that the indecomposable modules do not depend on the field, be it algebraically closed or not (see, for example, \cite[Theorem 1.23]{DDPW}). In this case also the (compactly generated) t-structures of $\cD(\bK Q)$ are independent of the field $\bK$. In particular, they depend on the lattice $\mathbf{Nc}(Q)$ of noncrossing partitions of $Q$ (see \cref{App} for more details).

\begin{prop*}[\cref{FiltNc}]
    For any field $\bK$ and Dynkin quiver $Q$, there is an order-preserving bijection between the lattices $\Aisle(\cD(\bK Q))$ and $\mathrm{Filt}(\mathbf{Nc}(Q))$.
\end{prop*}

In the following, $R$ will be a commutative noetherian ring and $Q$ a Dynkin quiver. We introduce a class of $RQ$-modules which, in a sense, play the role of the indecomposables in $\Mods(\bK Q)$.

\begin{defn}
    An $RQ$-module is called (\textit{free}) \textit{lattice} if it is vertexwise finitely generated projective (resp. free) over $R$. A lattice $X$ is called \textit{rigid} if $\Ext_{RQ}^1(X,X)=0$ and \textit{exceptional} if it is rigid and $\mathrm{End}_{RQ}(X)\cong R$. A lattice $X$ is said to have a \textit{rank vector} $d\in\bN^{\mid Q_0\mid}$ if for every $i\in Q_0$ and $\p\in\Spec(R)$ there is an isomorphism $\kp\otimes_R X_i\cong\kp^{d_i}$. Obviously, a free lattice has always a rank vector $d$ with $d_i=\rank(X_i)$.
\end{defn}

\begin{rem}\label{LatComp}
    By \cite[Lemma 3.1 (i)]{CB}, $RQ$-lattices have projective dimension less or equal then $1$. In particular, stalk complexes of $RQ$-lattices are compact in $\cD(RC)$. Moreover, by \cref{SmartRes}, they are $R$-K-projective. Thus, by \cref{DerAct}, any time the functor $\otimesL$ is applied to an $RQ$-lattice, it will be computed as the standard vertexwise tensor product.
\end{rem}

In view of the fact that, for any field $\bK$, the indecomposable $\bK Q$-modules are in bijection with the real Schur roots of $Q$, we can restate \cite[Theorem B]{CB} as follows. 
\begin{thm}\label{LatLift}
    For any field $\bK$, commutative ring $R$ and Dynkin quiver $Q$, there exists an indecomposable $\bK Q$-module $L$ with dimension vector $d$ if and only if there exists an exceptional $RQ$-lattice $X$ with rank vector $d$. In this case, there exists a unique exceptional free lattice $\widetilde{L}$ with rank vector $d$, and any rigid $RQ$-lattice of rank vector $d$ is exceptional.
\end{thm}

We will denote by $\widetilde{\quad}:\ind(\bK Q)\to\mathrm{lat}(RQ)$ the assignment in the theorem and we will refer to it as \textit{lattice lift}. Since the indecomposable objects of $\cD(\bK Q)$ are shifts of indecomposable $\bK Q$-modules, given a class of objects $\cX\subseteq\cD(\bK Q)$ closed under summands, by abuse of notation, we write
\[\widetilde{\cX}=\leftcurly\widetilde{L}[n]\mid L\in\ind(\bK Q)\text{ and }L[n]\in\cX\rightcurly\]
for the lattice lift of the indecomposables in $\cX$.

\begin{rem}\label{LatLiftRem}
We can actually give a concrete description of the lattice lifts. Indeed, by \cite{Rin}, for any field $\bK$, the exceptional $\bK Q$-modules are tree modules, i.e. they admit a presentation by $0,1$-matrices on the arrows. This applies to the indecomposable modules over Dynkin quivers, thus for an indecomposable $\bK Q$-module $L$ of dimension vector $d$, we can construct $\widetilde{L}$ defining $\widetilde{L}_i=R^{d_i}$ and $\widetilde{L}_\alpha=L_\alpha$ for any $i\in Q_0$ and $\alpha\in Q_1$. In this way, we obtain a free $RQ$-lattice which, by \cite[Lemma 3.1 (iv)]{CB}, is rigid and so, by previous theorem, exceptional. Thus, it is the unique lattice lift of $L$.
\end{rem}

The following proposition shows that the $\Hom$-orthogonality of two indecomposable $\bK Q$-modules is sufficient to detect the $\Hom$-orthogonality of their lattice lifts.
\begin{prop}\label{OrthProp}
    For any indecomposable $\bK Q$-modules $M$ and $L$, it holds that
    \begin{enumerate}
         \item $\Hom_{RQ}\left(\widetilde{M},\widetilde{L}\right)=0$ if and only if $\Hom_{\bK Q}(M,L)=0$;
        \item $\Ext_{RQ}^1\left(\widetilde{M},\widetilde{L}\right)=0$ if and only if $\Ext_{\bK Q}^1(M,L)=0$.
    \end{enumerate}
\begin{proof}
    By \cite[Theorem 4.1]{CB}, both the $R$-modules $\Hom_{RQ}\left(\widetilde{M},\widetilde{L}\right)$ and $\Ext_{RQ}^1\left(\widetilde{M},\widetilde{L}\right)$ are finitely generated projective of constant rank equal to the dimension over $\bK$ of the vector spaces $\Hom_{\bK Q}(M,L)$ and $\Ext_{\bK Q}(M,L)$ respectively. Thus, both point 1 and 2 hold.
\end{proof}
\end{prop}

\begin{cor}\label{OrthCor}
    For any two indecomposable $\bK Q$-complexes $M$ and $L$:
    \[\Hom_{\cD(RQ)}\left(\widetilde{M},\widetilde{L}\right)=0\text{ if and only if }\Hom_{\cD(\bK Q)}(M,L)=0\]
    In particular, for any class $\cX\subseteq\cD(\bK Q)$, it holds that $\widetilde{\cX^\perp}\subseteq\widetilde{\cX}^\perp$. 
\begin{proof}
    Let $M=M'[m]$ and $L=L'[\ell]$, where $M'$ and $L'$ are two indecomposable $\bK Q$-modules and $m,\ell\in\bZ$. There are standard isomorphisms
    \[\Hom_{\cD(RQ)}\left(\widetilde{M},\widetilde{L}\right)\cong\Ext_{RQ}^{\ell-m}\left(\widetilde{M'},\widetilde{L'}\right)\text{ and }\Hom_{\cD(\bK Q)}(M,L)\cong\Ext_{\bK Q}^{\ell-m}(M',L')\]
    By \cref{OrthProp}, the statement holds for $\ell-m=0,1$. Moreover, by \cref{LatComp} and since $\bK Q$ is hereditary, both $\Ext^{\ell-m}$ vanishes for $\ell-m<0$ and $\ell-m\geq2$. The particular part follows by taking $M\in\cX$ and $L\in\cX^\perp$.
\end{proof}
\end{cor}

\subsection{Compactly generated t-structures}
In this section we will introduce two assignments, linking the compactly generated aisles of $\cD(RQ)$ to the aisles of $\cD(\bK Q)$, which will lead to the complete classification of the former. In view of \cref{FiltNc}, we can refer to the poset $\Aisle(\cD(\bK Q))$ without specifying any field. The assignments are the following

\begin{equation}\begin{gathered}\label{AisleAss}
    \Aisle_\mathrm{cg}(\cD(RQ))\longleftrightarrow\Hom_\mathrm{Pos}(\Spec(R),\Aisle(\cD(\bK Q))) \\
    \varphi:\cU\longmapsto\left(\p\mapsto\aisle_{\bK Q}\leftangle  L\in\ind(\cD(\bK Q))\mid R/\p\otimesL\widetilde{L}\in\cU\rightangle\right) \\
    \aisle_{RQ}\leftangle R/\p\otimesL\widetilde{\sigma(\p)}\mid\p\in\Spec(R)\rightangle\longmapsfrom\,\sigma:\psi
\end{gathered}\end{equation}

\begin{lem}\label{AisleLift}
    For any ideals $\a$ and $\b$ and indecomposable $\bK Q$-complex $L$, it holds that:
    \begin{enumerate}
        \item If $\a\subseteq\b$ then $\aisle_{RQ}\leftangle R/\b\otimesL\widetilde{L}\rightangle\subseteq\aisle_{RQ}\leftangle R/\a\otimesL\widetilde{L}\rightangle$;
        \item If $\mathrm{rad}(\a)=\mathrm{rad}(\b)$ then $\aisle_{RQ}\leftangle R/\a\otimesL\widetilde{L}\rightangle=\aisle_{RQ}\leftangle R/\b\otimesL\widetilde{L}\rightangle$;
        \item $\aisle_{RQ}\leftangle R/\a\b\otimesL\widetilde{L}\rightangle=\aisle_{RQ}\leftangle R/\a\otimesL\widetilde{L}, R/\b\otimesL\widetilde{L}\rightangle$;
        \item $\aisle_{RQ}\leftangle R/\a\otimesL\widetilde{L}\rightangle=\aisle_{RQ}\leftangle K(\a)\otimesL\widetilde{L}\rightangle$ and it is compactly generated.
    \end{enumerate}
\begin{proof}
    (1) Let $\a\subseteq\b$, note that, similarly to \cref{ClosureLem}, the class
    \[\mathcal{M}=\leftcurly M\in\cD(R)\mid M\otimesL\widetilde{L}\in\aisle_{RQ}\leftangle R/\a\otimesL\widetilde{L}\rightangle\rightcurly\]
    is a cocomplete suspended subcategory containing $R/\a[0]$. By \cite[Proposition 2.4 (1)]{AJS}, $R/\b\in\aisle_R\leftangle R/\a\rightangle$, so $\mathcal{M}$ contains $R/\b[0]$.
    
    \noindent (2) Let $\mathrm{rad}(\a)=\mathrm{rad}(\b)$, as before, the class
    \[\mathcal{M'}=\leftcurly M\in\cD(R)\mid M\otimesL\widetilde{L}\in\aisle_{RQ}\leftangle R/\b\otimesL\widetilde{L}\rightangle\rightcurly\]
    it is a cocomplete suspended subcategory containing $R/\b[0]$ and, by \cite[Proposition 2.4 (4)]{AJS}, $\aisle_R\leftangle R/\a\rightangle=\aisle_R\leftangle R/\b\rightangle$. So, $R/\b[0]\in\mathcal{M}$ and $R/\a[0]\in\mathcal{M'}$ and the thesis follows.
    
    \noindent (3) The proof follows analogously to (2) since, by \cite[Proposition 2.4 (2)]{AJS}, $\aisle_R\leftangle R/\a\b\rightangle=\aisle_R\leftangle R/\a,R/\b\rightangle$.
    
    \noindent (4) The proof follows analogously to (2) since, by \cite[Lemma 5.3 (iii)]{Hrb}, $\aisle_R\leftangle R/\a\rightangle=\aisle_R\leftangle K(\a)\rightangle$. Moreover, since the complex $K(\a)$ is compact, by \cref{DerAct}, the functor $K(\a)\otimesL\_$ has a coproduct preserving right adjoint, and thus it preserves compactness by \cite[Theorem 5.1]{Nee}. Since, by \cref{LatComp}, any $RQ$-lattice is compact, the complex $K(\a)\otimesL\widetilde{L}$ is compact in $\cD(RQ)$.
\end{proof}
\end{lem}

\begin{prop}\label{InjAss}
The assignments \ref{AisleAss} are well defined and the composite $\varphi\circ\psi$ gives the identity on the poset $\Hom_\mathrm{Pos}(\Spec(R),\Aisle(\cD(\bK Q)))$. In particular, the assignment $\psi$ is injective.
\begin{proof}
    Clearly the assignment $\varphi$ gives a map from $\Spec(R)$ to $\Aisle(\bK Q)$ and, by \cref{AisleLift} (1), it is also a homomorphism of poset. Indeed, for any two prime ideals $\p\subseteq\q$, there is a containment
    \[\aisle_{\bK Q}\leftangle  L\mid R/\p\otimesL\widetilde{L}\in\cU\rightangle\subseteq\aisle_{\bK Q}\leftangle  L\mid R/\q\otimesL\widetilde{L}\in\cU\rightangle\]
    As for the assignment $\psi$, by \cref{AisleLift} (4), we have the equality
    \[\aisle_{RQ}\leftangle R/\p\otimesL\widetilde{\sigma(\p)}\rightangle=\aisle_{RQ}\leftangle K(\p)\otimesL\widetilde{\sigma(\p)}\rightangle\]
    and $K(\p)\otimesL\widetilde{\sigma(\p)}$ is compact, yielding a compactly generated aisle of $\cD(RQ)$.
    
    \noindent Denoting $\varphi\circ\psi(\sigma)$ by $\sigma'$, for a prime ideal $\p$ we have that
    \[\sigma'(\p)=\aisle_{\bK Q}\leftangle L\mid R/\p\otimesL\widetilde{L}\in\aisle_{RQ}\leftangle R/\q\otimesL\widetilde{\sigma(\q)}\mid\q\in\Spec(R)\rightangle\rightangle\]
    Obviously, $\sigma(\p)\subseteq\sigma'(\p)$. Let $L$ be a generator of $\sigma'(\p)$, by \cref{AisleLift} (4), we have that
    \[K(\p)\otimesL\widetilde{L}\in\aisle_{RQ}\leftangle K(\q)\otimesL\widetilde{\sigma(\q)}\mid\q\in\Spec(R)\rightangle\]
    By \cref{GenRem}, $K(\p)\otimesL\widetilde{L}$ is constructed trough extensions, positive shifts, direct summands and coproducts, from  $\leftcurly K(\q)\otimesL\widetilde{\sigma(\q)}\mid\q\in\Spec(R)\rightcurly$. Since all these operations commute with the base change $\kp\otimesL\_$, it holds that
    \[\kp\otimesL\left(K(\p)\otimesL\widetilde{L}\right)\in\aisle_{\kp Q}\leftangle\kp\otimesL\left(K(\q)\otimesL\widetilde{\sigma(\q)}\right)\mid\q\in\Spec(R)\rightangle\]
    then, by \cref{Koszul}, it follows that
    \[\kp\otimesL\widetilde{L}\in\aisle_{\kp Q}\leftangle\kp\otimesL\widetilde{\sigma(\q)}\mid\q\subseteq\p\rightangle\]
    Since aisles are independent of the field, we can conclude that $L\in\aisle_{\bK Q}\leftangle\sigma(\q)\mid\q\subseteq\p\rightangle$ and so, since $\sigma$ is a poset homomorphism, $L$ lies in $\sigma(\p)$.
\end{proof}
\end{prop}

In what follows we want to get a better understanding of the aisles of $\cD(RQ)$ reached by the assignment $\psi$ in \ref{AisleAss}. In particular, we will find a special family of compactly generated aisles from which all the other aisles are built from.
\begin{thm}\label{ElemAisle}
    For any specialization closed subset $V$ of $\Spec(R)$ and aisle $\cX$ of $\cD(\bK Q)$, there is a compactly generated t-structure
    \[\left(\cA_V^\cX,\cV_V^\cX\right):=\left(\aisle_{RQ}\leftangle R/\p\otimesL\widetilde{\cX}\mid\p\in V\rightangle,\leftcurly Y\in\cD(RQ)\mid\Gamma_V Y\in\widetilde{\cX}^\perp\rightcurly\right)\]
    Moreover, $\cA_V^\cX=\aisle_{RQ}\leftangle\widetilde{\cX}\rightangle\cap\supp_R^{-1}(V)$ and the associated truncation functor is $\tau_\cX^{<}\circ\Gamma_{V}$, where $\tau_\cX^{<}$ is the truncation functor of $\aisle_{RQ}\leftangle\widetilde{\cX}\rightangle$.
\begin{proof}
    Let $\cU_V^\cX=\aisle_{RQ}\leftangle\widetilde{\cX}\rightangle\cap\supp_R^{-1}(V)$, then we will prove that
    \begin{equation*}\tag{$\star$}\label{equalities}
        \cA_V^\cX=\cU_V^\cX={}^\perp\cV_V^\cX \,\text{ and }\, {\cA_V^\cX}^\perp={\cU_V^\cX}^\perp=\cV_V^\cX
    \end{equation*}
    The proof is divided into three steps:
    
    \noindent (i) We start proving that $\cA_V^\cX\subseteq\cU_V^\cX$. Fix a prime ideal $\p$ in $V$. By \cref{ClosureProp} (1), we have that
    \[R/\p\otimesL\widetilde{\cX}\subseteq\aisle_{RQ}\leftangle\widetilde{\cX}\rightangle\]
    In addition, since $R_\q\otimesL R/\p\neq0$ if and only if $\q\in V(\p)$, for any $\widetilde{L}$ in $\widetilde{\cX}$ and $\q\notin V(\p)$, it follows that
    \[\Gamma_\q(R/\p\otimesL\widetilde{L})=\Gamma_{V(\q)}\left((R_\q\otimesL R/\p)\otimesL\widetilde{L}\right)=0\]
    Thus we can conclude that $\supp_R(R/\p\otimesL\widetilde{L})\subseteq V(\p)$ which is contained in $V$.
    
    \noindent (ii) We show that $\Hom_{\cD(RQ)}\left(\cU_V^\cX,\cV_V^\cX\right)=0$. Let $X$ and $Y$ be complexes respectively in $\cU_V^\cX$ and $\cV_V^\cX$. If there exists a morphism $f:X\to Y$, then the functor $\Gamma_V$ induce the commutative square
    \[\begin{tikzcd} \Gamma_V X \arrow[r,"\gamma_X"] \arrow[d,"\Gamma_V f"] & X \arrow[d,"f"] \\ \Gamma_V Y \arrow[r,"\gamma_Y"] & Y \end{tikzcd}\]
    By \cref{SuppProp} (1), the morphism $\gamma_X$ is an isomorphism, thus $\Gamma_V X$ lies in the aisle generated by $\widetilde{\cX}$ and $\Gamma_V f$ is zero, so must be $f$.
    
    \noindent (iii) We prove the inclusion ${\cA_V^\cX}^\perp\subseteq\cV_V^\cX$. Recall the notation of \cref{GammaV}, write $V=\bigcup_{i\in I}V(\p_i)$ as a union of Zariski closed subsets, let $\mathcal{F}$ be the lattice of finite subsets of $I$ and, for any $F\in\mathcal{F}$, $\a_F=\prod\limits_{i\in F}\p_i$. For any $\widetilde{L}$ in $\widetilde{\cX}$ and $Y$ in ${\cA_V^\cX}^\perp$, since $\widetilde{L}$ is compact (\cref{LatComp}), by adjunction in \cref{DerAct}, we have that
    \[\Hom_{\cD(RQ)}\left(\widetilde{L},\Gamma_V Y\right)\cong{\underrightarrow{\lim}}_{(F,t)\in\mathcal{F}\times\bN}\Hom_{\cD(RQ)}\left(K\left(\a_F^{(t)}\right)\otimesL\widetilde{L},Y\right)\]
    Moreover, since the radical ideal of $\a_F^{(t)}$ is equal to the radical of $\a_F$, by \cref{AisleLift} (2)
    \[K\left(\a_F^{(t)}\right)\otimesL\widetilde{L}\in\aisle_{RQ}\leftangle R/\a_F\otimesL\widetilde{L}\rightangle=\aisle_{RQ}\leftangle R/\p_i\otimesL\widetilde{L}\mid i\in F\rightangle\subseteq\cA_V^\cX\]
    thus $\Hom_{\cD(RQ)}\left(\widetilde{L},\Gamma_V Y\right)=0$ and $\Gamma_V Y\in\widetilde{X}^\perp$.
    
    \noindent The first equality of (\ref{equalities}) follows by (i), (ii) and the inclusion ${}^\perp\cV_V^\cX\subseteq\cA_V^\cX$ from (iii); the second equality of (\ref{equalities}) follows by the inclusion ${\cU_V^\cX}^\perp\subseteq{\cA_V^\cX}^\perp$ from (i), (ii) and (iii).
    
    \noindent Let $\tau_\cX^<$ and $\tau_\cX^>$ be the truncation functors of $\aisle_{RQ}\leftangle\widetilde{X}\rightangle$ and $\widetilde{X}^\perp$, respectively. It remains to prove that the truncation functor associated to $\cA_V^\cX$ is the composite $\tau_\cX^{<}\circ\Gamma_V$. Consider the truncation triangle of $Z$ with respect to the t-structure $(\cA_V^\cX,\cV_V^\cX)$:
    \[\begin{tikzcd} X_Z \arrow[r,"f"] & Z \arrow[r,"g"] & Y_Z \arrow[r,"h"] & X_Z[1] \end{tikzcd}\]
    Since $\Gamma_V Y_Z$ is in $\widetilde{\cX}^\perp$, the morphism $\Gamma_V Z\xrightarrow{\Gamma_V(g)}\Gamma_V Y_Z$ factors trough $\tau_\cX^{>}(\Gamma_V Z)$, thus we can consider the octahedral diagram
    \[\begin{tikzcd}
    {\tau_\cX^{<}(\Gamma_V Z)} \arrow[d, dashed] \arrow[r] & \Gamma_V Z \arrow[r] \arrow[d,"\|", phantom] & \tau_\cX^{>}(\Gamma_V Z) \arrow[d] \arrow[r,"+"] & {} \\
    {\Gamma_V X_Z} \arrow[d, dashed] \arrow[r] & \Gamma_V Z \arrow[r,"{\Gamma_V(g)}"] \arrow[d] & \Gamma_V Y_Z \arrow[d,"\|", phantom] \arrow[r,"+"] & {} \\
    C \arrow[r] & \tau_\cX^{>}(\Gamma_V Z) \arrow[r] & \Gamma_V Y_Z \arrow[r,"+"] & {}
    \end{tikzcd}\]
    According to the octahedral axiom, the first column is a distinguished triangle. So is its rotation
    \[C[-1]\to\tau_\cX^{<}(\Gamma_V Z)\to\Gamma_V X_Z\xrightarrow{+}C\]
    Noting that, from the third row of the octahedron, $C$ is in $\widetilde{\cX}^\perp$ and $\Gamma_V X_Z\cong X_Z$ is in $\aisle_{RQ}\leftangle\widetilde{\cX}\rightangle$, the triangle splits, i.e. $\tau_\cX^{<}(\Gamma_V Z)\cong C[-1]\oplus X_Z$. This forces $C$ to be also in $\aisle_{RQ}\leftangle\widetilde{X}\rightangle$, thus it is zero and $\tau_\cX^{<}(\Gamma_V Z)\cong X_Z$.
\end{proof}
\end{thm}

Given a homomorphism $\sigma:\Spec(R)\to\Aisle(\cD(\bK Q))$ its associated aisle is
\[\cA_\sigma:=\aisle_{RQ}\leftangle R/\p\otimesL\widetilde{\sigma(\p)}\mid\p\in\Spec(R)\rightangle=\aisle_{RQ}\leftangle\bigcup_{\Spec(R)}\cA_{V(\p)}^{\sigma(\p)}\rightangle\]
In particular, the aisles of the form $\cA_{V(\p)}^{\sigma(\p)}$ are a special kind of aisles, from which we can construct the ones given by $\psi$ in \ref{AisleAss}. Note that, by \cref{ElemAisle}, we can compute also the associated (homotopically smashing) coaisle $\cV_\sigma:=\cA_\sigma^\perp$, indeed
\[\cV_\sigma=\bigcap_{\Spec(R)}\cV_{V(\p)}^{\sigma(\p)}=\leftcurly Y\in\cD(RQ)\mid\Gamma_{V(\p)}Y\in\widetilde{\sigma(\p)}^\perp\text{ for any }\p\in\Spec(R)\rightcurly\]
Note that in general the intersection of coaisles is not a coaisle, but this holds for homotopically smashing ones since these are defined just by closure operators. Indeed, by \cite[Remark 8.9]{SS} and \cite[Theorem 4.7, Proposition 5.10]{LV}, they are exactly the homotopically smashing complete cosuspended subcategories closed under pure subobjects (we refer to \textit{loc. cit.} for the undefined terminology).

A parallel approach that would allow a classification of the compactly generated t-structures is to classify the homotopically smashing ones and then prove that these are the same, i.e. prove that the telescope conjecture holds.

\subsection{Homotopically smashing t-structures}
In this section we will show that, according to the commutative setting (\cite[Section 2]{HN}), homotopically smashing coaisles are cogenerated by stalk complexes of vertexwise injective $RQ$-modules. In particular, given a poset homomorphism $\sigma:\Spec(R)\to\Aisle(\bK Q)$, a candidate set to cogenerate the coaisle $\cV_\sigma$ is
\[\mathcal{E}_\sigma=\leftcurly E(R/\p)\otimesL\widetilde{\sigma(\p)^\perp}\mid\p\in\Spec(R)\rightcurly\]

The following is a dual version of \cref{AisleLift} and shows how, for the generation of coaisles, it is equivalent to consider modules of the form $E(R/\p)\otimesL\widetilde{L}$ or $\kp\otimesL\widetilde{L}$.
\begin{prop}\label{CoaisleLift}
    For any prime ideals $\p$ and $\q$ and indecomposable $\bK Q$-complex $L$, it holds that:
    \begin{enumerate}
        \item If $\q\subseteq\p$ then $\cosusp_{RQ}^{\Pi,\hs}\leftangle E(R/\q)\otimesL\widetilde{L}\rightangle\subseteq\cosusp_{RQ}^{\Pi,\hs}\leftangle E(R/\p)\otimesL\widetilde{L}\rightangle$;
        \item $\cosusp_{RQ}^{\Pi,\hs}\leftangle E(R/\p)\otimesL\widetilde{L}\rightangle=\cosusp_{RQ}^{\Pi,\hs}\leftangle \kp\otimesL\widetilde{L}\rightangle$.
    \end{enumerate}
\begin{proof}
    Firstly note that the proofs of \cite[Lemma 2.6, Lemma 2.8, Lemma 2.9]{HN} holds true even for $\cV$ homotopically smashing complete cosuspended subcategory of $\cD(R)$, without assuming it is a coaisle. In particular, for such $\cV\subseteq\cD(R)$ and for any $n\in\bZ$, it holds that
    \[E(R/\p)[-n]\in\cV\text{ if and only if }\kp[-n]\in\cV\]
    and, for any prime ideal $\q\subseteq\p$, that
    \[E(R/\p)[-n]\in\cV\text{ implies }E(R/\q)[-n]\in\cV\]  
    
    \noindent (1) Let $\q\subseteq\p$, consider the class
    \[\mathcal{M}=\leftcurly M\in\cD(R)\mid M\otimesL\widetilde{L}\in\cosusp_{RQ}^{\Pi,\hs}\leftangle E(R/\p)\otimesL\widetilde{L}\rightangle\rightcurly\]
    it is a homotopically smashing complete cosuspended subcategory of $\cD(R)$ containing $E(R/\p)[0]$, thus, by the above discussion, it contains also $E(R/\q)[0]$.
    
    \noindent (2) Consider the class
    \[\mathcal{M}'=\leftcurly M\in\cD(R)\mid M\otimesL\widetilde{L}\in\cosusp_{RQ}^{\Pi,\hs}\leftangle \kp\otimesL\widetilde{L}\rightangle\rightcurly\]
    it is a homotopically smashing complete cosuspended subcategory of $\cD(R)$ containing $\kp[0]$. By the above discussion, $\kp[0]\in\mathcal{M}$ and $E(R/\p)[0]\in\mathcal{M}'$ and the thesis follows.
\end{proof}
\end{prop}

\begin{cor}\label{CoaisleCor}
    Let $\cV$ be a homotopically smashing coaisle of $\cD(RQ)$, $L$ an indecomposable $\bK Q$-complex and $\p\in\Spec(R)$, then:
    \begin{enumerate}
        \item For any prime ideal $\q\subseteq\p$, $E(R/\p)\otimesL\widetilde{L}\in\cV$ implies $E(R/\q)\otimesL\widetilde{L}\in\cV$;
        \item $E(R/\p)\otimesL\widetilde{L}\in\cV$ if and only if $\kp\otimesL\widetilde{L}\in\cV$.
    \end{enumerate}
\end{cor}

Now we present a slight variation of \cref{OrthProp}, showing how the $\Hom$-orthogonality of two indecomposable $\bK Q$-modules $M$ and $L$ is sufficient to detect the $\Hom$-orthogonality in $\cD(RQ)$ of $\widetilde{M}$ and $\kp\otimesL\widetilde{L}$.

\begin{lem}
    Let $M$ and $L$ be two indecomposable $\bK Q$-modules. Then, for any $\p\in\Spec(R)$, the following holds
    \begin{enumerate}
        \item $\Hom_{RQ}\left(\widetilde{M},\kp\otimesL\widetilde{L}\right)\neq0$ if and only if $\Hom_{\bK Q}(M,L)\neq0$;
        \item $\Ext_{RQ}^1\left(\widetilde{M},\kp\otimesL\widetilde{L}\right)\neq0$ if and only if $\Ext_{\bK Q}^1(M,L)\neq0$.
    \end{enumerate}
\begin{proof}
    (1) Suppose that $\Hom_{\bK Q}(M,L)\neq0$. By \cite[Theorem 4.1 (ii)]{CB} applied to the ring homomorphism $\pi:R\to\kp$, the natural morphism $\kp\otimes_R\Hom_{RQ}\left(\widetilde{M},\widetilde{L}\right)\to\Hom_{\kp Q}(M,L)$ is an isomorphism. Then, if the latter vector space is nonzero, we can find a nonzero homomorphism $f:M\to L$, which is the image of $1_{\kp}\otimes\widetilde{f}$, for a nonzero homomorphism $\widetilde{f}:\widetilde{M}\to\widetilde{L}$. In particular, $\widetilde{f}$ should send some nonzero element of $\widetilde{M}$ to an element of $\widetilde{L}$ which is not annihilated by the tensor product with $\kp$. Thus, the composite $\begin{tikzcd} \widetilde{M} \arrow[r,"\widetilde{f}"] & \widetilde{L} \arrow[r,"\pi\otimes\id_L"] & \kp\otimesL\widetilde{L} \end{tikzcd}$ is a nonzero homomorphism.
    
    For the other direction, take a nonzero homomorphism $\overline{f}:\widetilde{M}\to\kp\otimesL\widetilde{L}$. Note that the embedding $\Mods(\kp Q)\hookrightarrow\Mods(RQ)$ can be identified with the action $\kp\otimes_{\kp}\_$ and so, by \cref{FunProp} (2), it is right adjoint to $\kp\otimes_R\_$. By \cite[Corollary 3.2]{RSV}, the morphism $\overline{f}$ factors through $\kp\otimes_R\widetilde{M}$ giving a nonzero morphism $f:\kp\otimesL\widetilde{M}\to\kp\otimesL\widetilde{L}$ in $\Mods(\kp Q)$. Thus, by \cref{OrthProp} (taking $R=\kp$), we have that $\Hom_{\bK Q}(M,L)\neq0$.
    
    \noindent (2) Suppose that $\Ext_{\bK Q}^1(M,L)\neq0$. By \cite[Section 3]{Rin} (see also \cite[Lemma 3.8]{Wei}), it has a basis whose elements are extensions of the form
    \[\begin{tikzcd} 0 \arrow[r] & L \arrow[r,"\begin{bmatrix} \id_L \\ 0 \end{bmatrix}"] & E \arrow[r,"\begin{bmatrix} 0\,\id_M \end{bmatrix}"] & M \arrow[r] & 0 \end{tikzcd} \text{where } E=\left(L_i\oplus M_i,\begin{bmatrix} L_\alpha & \varepsilon_\alpha \\ 0 & M_\alpha \end{bmatrix}\right)\]
    with $\varepsilon_\alpha$ represented by $0,1$-matrices, as $L_\alpha$ and $M_\alpha$ (see \cref{LatLiftRem}). We can choose all such matrices over a field $\bK$ of characteristic $0$ to guarantee that all the commutativity relation between them keep holding over an arbitrary ring, and then consider the following extension
    \[\begin{tikzcd} 0 \arrow[r] & \kp\!\otimesL\!\widetilde{L} \arrow[r,"\begin{bmatrix} \id_L \\ 0 \end{bmatrix}"] & \overline{E} \arrow[r,"\begin{bmatrix} 0\,\id_{\widetilde{M}} \end{bmatrix}"] & \widetilde{M} \arrow[r] & 0 \end{tikzcd} \text{where } \overline{E}=\left(\left(\kp\!\otimesL\!\widetilde{L}_i\right)\oplus \widetilde{M}_i,\begin{bmatrix} L_\alpha & \varepsilon_\alpha\!\cdot\!\pi \\ 0 & M_\alpha \end{bmatrix}\right)\]
    Note that the sequence is exact since it is vertexwise split and it does not split over $RQ$ because at least one $\varepsilon_\alpha$ is non zero.
    
    For the other direction, consider a nontrivial extension $0\to\kp\otimesL\widetilde{L}\to P\to\widetilde{M}\to 0$. Since $\widetilde{M}$ is vertexwise free, the sequence is vertexwise split. Thus, applying the functor $\kp\otimes_R\_$ to it, we get the following diagram with exact rows
    \begin{equation*}
        \begin{tikzcd} 0 \arrow[r] & \kp\otimesL\widetilde{L} \arrow[r, "f"] \arrow[d, "\|", phantom] & P \arrow[r] \arrow[d, "\varphi"] & \widetilde{M} \arrow[r] \arrow[d] & 0 \\
        0 \arrow[r] & \kp\otimesL\widetilde{L} \arrow[r, "h"] & \kp\otimes_R P \arrow[r] & \kp\otimes_R\widetilde{M} \arrow[r] & 0 \end{tikzcd}
    \end{equation*}
    Note that if the second row were split, there would be a map $e:\kp\otimes_R P\to\kp\otimesL\widetilde{L}$ such that $e\circ\varphi\circ f=e\circ h=\id_L$. In particular, the first row would split too. Thus we get a nontrivial extension in $\Ext_{RQ}^1\left(\widetilde{M},\kp\otimesL\widetilde{L}\right)$.
\end{proof}
\end{lem}

\begin{cor}\label{VarOrth}
    For any two indecomposable $\bK Q$-complexes $M$ and $L$, and for any $\p\in\Spec(R)$: 
    \[\Hom_{\cD(RQ)}\left(\widetilde{M},\kp\otimesL\widetilde{L}\right)=0\text{ if and only if }\Hom_{\cD(\bK Q)}(M,L)=0\]
\begin{proof}
    Analogous to \cref{OrthCor}.
\end{proof}
\end{cor}

The next theorem will prove that the modules of the form $E(R/\p)\otimesL\widetilde{L}$ which are contained in $\cV_\sigma$ are exactly the ones in $\mathcal{E}_\sigma$.
\begin{thm}\label{Esigma}
    Let $\sigma:\Spec(R)\to\Aisle(\cD(\bK Q))$ be a homomorphism of posets, for any prime ideal $\p$ and indecomposable $\bK Q$-complex $L$, then
    \[E(R/\p)\otimesL\widetilde{L}\in\cV_\sigma\text{ if and only if }L\in\sigma(\p)^\perp\]
\begin{proof}
    Recall that $\cV_\sigma=\cA_\sigma^\perp$ where $\cA_\sigma:=\aisle_{RQ}\leftangle R/\q\otimesL\widetilde{\sigma(\q)}\mid\q\in\Spec(R)\rightangle$. Let $L$ be an indecomposable $\bK Q$-complex.
    By \cref{AisleLift} (4) and \cref{CoaisleCor} (2), $E(R/\p)\otimesL\widetilde{L}\in\cA_\sigma^\perp$ if and only if
    \[\Hom_{\cD(RQ)}\left(K(\q)\otimesL\widetilde{M},\kp\otimesL\widetilde{L}\right)=0\] for any $\q\in\Spec(R)$ and any $M\in\sigma(\q)$. By adjunction (\cref{DerAct}) and self-duality of the Koszul complex (\cref{Koszul}), this holds if and only if
    \[\Hom_{\cD(RQ)}\left(\widetilde{M},\left(\kp\otimesL K(\q)[-n]\right)\otimesL\widetilde{L}\right)=0\]
    for any $\q\in\Spec(R)$ and any $M\in\sigma(\q)$. By \cref{Koszul}, this is always true for $\q\nsubseteq\p$, while for $\q\subseteq\p$ it is equivalent to
    \[\bigoplus_{i=0}^n\Hom_{\cD(RQ)}\left(\widetilde{M},\kp^{\binom{n}{i}}[i-n]\otimesL\widetilde{L}\right)=0\]
    for any $M\in\sigma(\q)$. Since $\sigma(\q)$ is closed under positive shifts, this is equivalent to
    \[\Hom_{\cD(RQ)}\left(\widetilde{M},\kp\otimesL\widetilde{L}\right)=0\]
    for any $M\in\sigma(\q)$. Thus, by \cref{VarOrth}, we can conclude that $E(R/\p)\otimesL\widetilde{L}\in\cA_\sigma^\perp$ if and only if $\Hom_{\cD(\bK Q)}(M,L)=0$ for any $M\in\sigma(\q)$ and $\q\subseteq\p$. Which means, since $\sigma$ is a poset homomorphism, for any $M\in\sigma(\p)$. So, the thesis follows.
\end{proof}
\end{thm}

\begin{thm}\label{Cogen}
    For any homotopically smashing coaisle $\cV\subseteq\cD(RQ)$:
    \[\cV=\cosusp_{RQ}^{\Pi,\hs}\leftangle E(R/\p)\otimesL\widetilde{L}\in\cV\mid\p\in\Spec(R)\rightangle\]
\begin{proof}
    Denote the right hand side by $\mathcal{W}$. The containment $\mathcal{W}\subseteq\cV$ is clear. On the other hand, if $Y\in\cV$, by \cref{ClosureProp} (3), for any prime ideal $\p$ the complex $\Gamma_\p Y$ lies in $\cV$ and so does $\RHom_{R}(\kp,\Gamma_\p Y)$. The latter complex is a direct sum of indecomposable $\kp Q$-complexes $L_i$ such that $L_i\cong\kp\otimesL\widetilde{L_i}$ and they are all contained in $\cV$. Since, by \cref{CoaisleLift} (2) and \cref{CoaisleCor} (2), $\cW=\cosusp_{RQ}^{\Pi,\hs}\leftangle\kp\otimesL\widetilde{L}\in\cV\mid\p\in\Spec(R)\rightangle$, the complex $\RHom_{R}(\kp,\Gamma_\p Y)$ lies in $\cW$. Thus, by minimality (\cref{Min}), $\Gamma_\p Y\in\mathcal{W}$ for any prime ideal $\p$ and, by local to global principle (\cref{LTG}), $Y$ lies in $\mathcal{W}$.
\end{proof}
\end{thm}

\section{Main results}\label{MainRes}
\subsection{Telescope conjecture and Classification}
Thanks to these last results, we can build an assignment which links homotopically smashing coaisles of $\cD(RQ)$ to coaisles of $\cD(\bK Q)$. Note that we will consider the poset $\mathrm{Coaisle}(\cD(\bK Q))$ with the reverse order, which is actually isomorphic to $\Aisle(\cD(\bK Q))$. The assignment is the following:
\begin{equation*}\begin{gathered}
    \omega:\mathrm{Coaisle_{hs}}(\cD(RQ))\longrightarrow\Hom_\mathrm{Pos}(\Spec(R),\mathrm{Coaisle}(\cD(\bK Q))^\mathrm{op}) \\
    \cV\longmapsto\left(\omega_\cV:\p\mapsto\cosusp_{\bK Q}^{\Pi,\hs}\leftangle L\in\ind(\cD(\bK Q))\mid E(R/\p)\otimesL\widetilde{L}\in\cV\rightangle\right)
\end{gathered}\end{equation*}

\begin{rem}\label{CoaisleAss}
    The assignment $\omega$ is well-defined and injective.
    \begin{itemize}
        \item By \cref{GenRem}, for any $\p\in\Spec(R)$, $\omega_\cV(\p)$ is a (homotopically smashing) coaisle.
        \item The map $\omega_\cV$ is a poset homomorphism. Indeed, for any prime ideal $\q\subseteq\p$, by \cref{CoaisleLift} (1) we have that
        \[\cosusp_{\bK Q}^{\Pi,\hs}\leftangle L\mid E(R/\p)\otimesL\widetilde{L}\in\cV\rightangle\subseteq\cosusp_{\bK Q}^{\Pi,\hs}\leftangle L\mid E(R/\q)\otimesL\widetilde{L}\in\cV\rightangle\]
        i.e. $\omega_\cV(\q)\leq\omega_\cV(\p)$ in $\mathrm{Coaisle}(\cD(\bK Q))^\mathrm{op}$.
        \item The assignment $\omega$ is injective. Suppose that given two homotopically smashing coaisles $\cV$ and $\cV'$ we have that $\omega_\cV=\omega_{\cV'}$, then by \cref{Cogen} it turns out that
        \[\cV=\cosusp_{RQ}^{\Pi,\hs}\leftangle E(R/\p)\otimesL\widetilde{\omega_\cV(\p)}\rightangle=\cV'\]
    \end{itemize}
\end{rem}

\begin{thm}\label{Main}
    Let $R$ be a commutative noetherian ring and $Q$ a Dynkin quiver, then:
    \begin{enumerate}
        \item The following assignments are mutually inverse poset isomorphisms
        \begin{equation*}\begin{gathered}
            \Aisle_\mathrm{cg}(\cD(RQ))\longleftrightarrow\Hom_\mathrm{Pos}(\Spec(R),\Aisle(\cD(\bK Q))) \\
            \varphi:\cU\longmapsto\left(\p\mapsto\aisle_{\bK Q}\leftangle L\in\ind(\cD(\bK Q))\mid R/\p\otimesL\widetilde{L}\in\cU\rightangle\right) \\
            \aisle_{RQ}\leftangle R/\p\otimesL\widetilde{\sigma(\p)}\mid\p\in\Spec(R)\rightangle\longmapsfrom\,\sigma:\psi
        \end{gathered}\end{equation*}
        \item Any homotopically smashing t-structure in $\cD(RQ)$ is compactly generated.
    \end{enumerate}
\begin{proof} 
    (1) By \cref{InjAss}, it suffices to prove that $\psi$ is surjective. Given a compactly generated aisle $\cA$ in $\cD(RQ)$, the associated coaisle $\cV:=\cA^\perp$ is homotopically smashing. From the above discussion we can associate to it the poset homomorphism $\omega_\cV:\Spec(R)\to\mathrm{Coaisle}(\cD(\bK Q))^\mathrm{op}$. Define $\sigma\in\Hom_\mathrm{Pos}(\Spec(R),\Aisle(\cD(\bK Q)))$ to be such that $\sigma(\p):={}^\perp\omega_\cV(\p)$ for any prime ideal $\p$, and consider the compactly generated t-structure $(\cA_\sigma,\cV_\sigma)$, where $\cA_\sigma=\psi(\sigma)$. By \cref{Esigma}, $\omega_{\cV_\sigma}(\p)=\sigma(\p)^\perp=\omega_\cV(\p)$ for any prime ideal $\p$. Thus the homomorphism $\omega_{\cV_\sigma}$ is precisely $\omega_\cV$, and so, by injectivity of $\omega$, $(\cA,\cV)=(\cA_\sigma,\cV_\sigma)$.
    
    \noindent (2) Let $(\cA,\cV)$ be a homotopically smashing t-structure. Following the proof of point (1), we conclude that $(\cA,\cV)=(\cA_\sigma,\cV_\sigma)$ for some homomorphism $\sigma$. Thus, it is compactly generated.
\end{proof}
\end{thm}

\subsection{Wide subcategories}
Finally, we want to end the section by proving that, thanks to \cref{Cogen}, compactly generated aisles are determined on cohomology. This was already known both for hereditary finite dimensional algebras (by hereditary property) and for commutative noetherian ring (by \cite{AJS}).

\begin{lem}\label{CohDetLem}
    For any prime ideal $\p$ and indecomposable $\bK Q$-module $L$, the injective dimension of $E(R/\p)\otimesL\widetilde{L}$ over $RQ$ is less or equal then $1$.
\begin{proof}
    Let $\p\in\Spec(R)$ and $L\in\ind(\bK Q)$, consider the $RQ$-module $E(R/\p)\otimesL\widetilde{L}$ and a minimal injective resolution $E$ over $RQ$. Note that $\supp_R(E)=\supp_R\left(E(R/\p)\otimesL\widetilde{L}\right)=\{\p\}$ (see \cite[15.1.12]{CFH}). In particular, there is a long exact sequence
    \begin{equation*}\tag{$\ast$}\label{InjRes}
        \begin{tikzcd} 0 \arrow[r] & E(R/\p)\otimesL\widetilde{L} \arrow[r] & E^0 \arrow[r] & E^1 \arrow[r] & E^2 \arrow[r] & \ldots \end{tikzcd}
    \end{equation*}
    where, for all $n\geq0$, $E^n$ is an injective $RQ$-module supported on $\{\p\}$, i.e. it is vertexwise isomorphic to a direct sum of copies of $E(R/\p)$ and for any $i\in Q_0$ the map
    \[E_i^n\to\bigoplus_{\substack{\alpha\in Q_1 \\ s(\alpha)=i}}E_{t(\alpha)}^n\]
    is a split surjection (see \cite[Theorem 4.2]{EEG}). Since $\Hom_R(\kp,E(R/\p))\cong\kp$ and $\Hom_R(\kp,\_)$ preserves splitting, by \cite[Theorem 4.2]{EEG}, $\Hom_R(\kp,E^n)$ is an injective $\kp Q$-module for any $n$. Moreover, since all the maps in (\ref{InjRes}) are vertexwise split, applying $\Hom_R(\kp,\_)$ provides a minimal injective resolution of $\Hom_R(\kp,E(R/\p)\otimesL\widetilde{L})\cong\kp\otimesL\widetilde{L}$. Since $\kp Q$ is hereditary, $\kp\otimesL\widetilde{L}$ has injective dimension less or equal then $1$ and, by minimality of (\ref{InjRes}), we can conclude that $E^n=0$ for $n\geq2$.
\end{proof}
\end{lem}

\begin{thm}\label{CohDetThm}
    Let $(\cU,\cV)$ be a t-structure in the derived category of a ring $\cD(A)$. If $\cU={}^\perp\mathcal{E}$ for a set $\mathcal{E}$ of $A$-modules of injective dimension at most 1, then $\cU$ is determined on cohomology, i.e. for any $X\in\cD(RQ)$, it holds that
    \[X\in\cU\text{ if and only if }H^i(X)[-i]\in\cU\text{ for any }i\in\bZ\]
\begin{proof}
    Let $E\cong E_0\xrightarrow{d}E_1$ be a complex of injective $A$-modules concentrated in degrees 0 and 1. This induce the distinguished triangle
    \[\begin{tikzcd} E_1[-1] \arrow[r] & E \arrow[r] & E_0[0] \arrow[r,"d"] & E_1[0] \end{tikzcd}\]
    Given $X\in\cD(A)$, $X\in\cU$ if and only if $\Hom_{\cD(A)}(X,E[i])=0$ for any $i\leq0$ if and only if $H^i(\RHom_A(X,E))=0$ for any $i\leq0$. Noting that for $k=0,1$ there is an isomorphism $H^i(\RHom_A(X,E_k))\cong\Hom_A(H^{-i}(X),E_k)$ for any $i\leq0$, the long exact sequence in cohomology of the complex $\RHom_A(X,E)$ looks as
    \begin{equation*}\begin{aligned}
        \ldots&\to{\Hom_A(H^1(X),E_0)}\xrightarrow{\Hom_A(H^1(X),d)}{\Hom_A(H^1(X),E_1)}\to{\Hom_{\cD(A)}(X,E)}\longrightarrow{} \\
        {}&\longrightarrow{\Hom_A(H^0(X),E_0)}\xrightarrow{\Hom_A(H^0(X),d)}{\Hom_A(H^0(X),E_1)}\to\ldots
    \end{aligned}\end{equation*}
    It follows that $\Hom_{\cD(A)}(X,E[i])=0$ for any $i\leq0$ if and only if $\Hom_A(H^0(X),d)$ is injective and $\Hom_A(H^i(X),d)$ is an isomorphism for any $i>0$. Consider the complex $\overline{X}=\bigoplus_{i\in\bZ}H^i(X)[-i]$ and note that these latter conditions are equivalent to $\Hom_{\cD(A)}\left(\overline{X},E[i]\right)=0$ for any $i\leq0$, since the long exact sequence in cohomology would be the same. We can conclude that $X\in\cU$ if and only if $\overline{X}\in\cU$ and so the statement follows.
\end{proof}
\end{thm}

\begin{cor}\label{DetCoh}
    Compactly generated aisles of $\cD(RQ)$ are determined on cohomology.
\begin{proof}
    Since any compactly generated t-structure of $\cD(RQ)$ is homotopically smashing, by \cref{Cogen} and \cref{CohDetLem}, the aisle is left orthogonal to set of $RQ$-modules of injective dimension at most 1. Then \cref{CohDetThm} concludes the proof.
\end{proof}
\end{cor}

As an application of the previous theorem we can provide a classification of the wide subcategories of $\mods(RQ)$, the category of finitely generated modules over $RQ$. This classification partially merges the one over commutative noetherian rings due to Takahashi in \cite{Tak} and the one over representation-finite hereditary algebras due to Ingalls and Thomas in \cite{IT}. This classification holds for rings $RQ$ with $R$ commutative noetherian regular and $Q$ Dynkin.

\begin{thm}\label{Wide}
    Let $R$ be a commutative noetherian regular ring and $Q$ a Dynkin quiver. Then, the following maps form a bijection
    \begin{equation*}\begin{gathered}
        \mathrm{Wide}(RQ)\longleftrightarrow\Hom_\mathrm{Pos}(\Spec(R),\mathrm{Wide}(\bK Q)) \\
        \cW\longmapsto\left(\p\mapsto\wide_{\bK Q}\leftangle L\in\ind(\bK Q)\mid R/\p\otimesL\widetilde{L}\in\cW\rightangle\right) \\
        \wide_{RQ}\leftangle R/\p\otimesL\widetilde{\sigma(\p)}\mid\p\in\Spec(R)\rightangle\longmapsfrom\,\sigma
    \end{gathered}\end{equation*}
\begin{proof}
    By regularity of $R$, we get that $\cD^c(RQ)=\cD^b(\mods(RQ))$ \cite[20.2.11]{CFH} and, by the (stable) telescope conjecture, every thick subcategory $\cT\subseteq\cD^c(RQ)$ is equal to $\loc_{RQ}\leftangle\cT\rightangle\cap\cD^c(RQ)$, and it is determined on cohomology by \cref{DetCoh}. Thus, by \cite[Theorem 2.5]{ZC}, there is a bijection
    \begin{equation*}\tag{$\ast\ast$}\label{ZCbij}\begin{gathered}
        \mathrm{Wide}(RQ)\longleftrightarrow\mathrm{Thick}(\cD^c(RQ)) \\
        \cW\longmapsto\leftcurly X\in\cD^c(\bK Q)\mid H^i(X)\in\cW\text{ for any }i\in\bZ\rightcurly \\
        \leftcurly H^0(X)\mid X\in\cT\rightcurly\longmapsfrom\,\cT
    \end{gathered}\end{equation*}
    Note that, by virtue of the telescope conjecture and the classification in \cref{Main}, the lattice $\mathrm{Thick}(\cD^c(RQ))$ is in bijection with $\Hom_\mathrm{Pos}(\Spec(R),\mathrm{Thick}(\cD^c(\bK Q)))$. Thus, specifying the above to $R=\bK$ (see also \cite[Theorem 5.1]{Bru}) and composing all the maps we obtain the bijection
    \begin{equation*}\begin{gathered}
        \mathrm{Wide}(RQ)\longleftrightarrow\Hom_\mathrm{Pos}(\Spec(R),\mathrm{Wide}(\bK Q)) \\
        \cW\longmapsto\left(\p\mapsto\leftcurly H^0(X)\mid X\in\cT_\cW\rightcurly\right) \\
        \leftcurly H^0(X)\mid X\in\cT_\sigma\rightcurly\longmapsfrom\,\sigma
    \end{gathered}\end{equation*}
    where $\cT_\cW=\thick_{\bK Q}\leftangle L\mid R/\p\otimes_R^{\bf L}\widetilde{L}\in\cW\rightangle$ and $\cT_\sigma=\thick_{RQ}\leftangle R/\p\otimes_R^{\bf L}\widetilde{\sigma(\p)}\mid\p\in\Spec(R)\rightangle$. Since both $\cT_\cW$ and $\cT_\sigma$ are cohomology determined and generated by compacts, by \cite[Lemma 5.1]{Hrb}, for any complex $X$ in there, $H^0(X)$ is a direct summand of an $n$-fold extension of finite direct sums of shifts of its generators. Thus, by the bijection (\ref{ZCbij}), the assignments coincides with the one in the statements.
\end{proof}
\end{thm}

Note that, by \cite[Proposition 2.3, Theorem 4.6]{ASe} and \cref{Wide}, all the wide subcategories of the category $\mods{RQ}$ will be of the form $\alpha\left(\mathsf{filtgen}\leftangle R/\p\otimesL\widetilde{\sigma(\p)}\rightangle\right)$ for some homomorphism $\sigma:\Spec(R)\to\mathrm{Wide}(\bK Q)$ (where \textsf{filtgen} denotes the closure under quotients and extensions, while for the definition of the map $\alpha$ we refer to \cref{ITRem}).

\begin{exmp}
    We can compute the lattice of wide subcategories of the path algebra $\mathbb{C}[[x]]A_2$. Recall that
    \[\Spec{\mathbb{C}[[x]]}=\begin{tikzcd} (x) \\  \\ (0) \arrow[uu]\end{tikzcd}\quad\text{and}\quad\mathrm{Wide}(\bK A_2)=\begin{tikzcd} & \mods(\bK A_2) & \\ \leftangle 0\to\bK \rightangle \arrow[ru] & \leftangle \bK\to\bK \rightangle \arrow[u] & \leftangle \bK\to0 \rightangle \arrow[lu] \\ & 0\to 0 \arrow[lu] \arrow[ru] \arrow[u] & \end{tikzcd}\]
    Then, denoting by $M\to M'$ a general $\mathbb{C}[[x]]A_2$-module and by $T$ and $T'$ the torsion $\mathbb{C}[[x]]$-modules, the lattice $\mathrm{Wide}(\mathbb{C}[[x]]A_2)$ is
    \[\begin{tikzcd} &  & M\to M' &  & \\ & T\to M \arrow[ru] & M\xrightarrow{f_{tor}}M' \arrow[u] & M\to T \arrow[lu] & \\ 0\to M \arrow[ru] & M\xrightarrow{\cong}M' \arrow[ru, bend left] & T\to T' \arrow[ru] \arrow[lu] \arrow[u] &  & M\to 0 \arrow[lu] \\ & 0\to T \arrow[ru] \arrow[lu] & T\xrightarrow{\cong}T' \arrow[u] \arrow[lu, bend left] & T\to 0 \arrow[ru] \arrow[lu] & \\ &  & 0\to 0 \arrow[ru] \arrow[lu] \arrow[u] &  & \end{tikzcd}\]
    where $f_{tor}$ is a $\mathbb{C}[[x]]$-linear homomorphism with torsion kernel and cokernel and in every vertex we consider the set of all the representations of that kind. For example, taking the homomorphism $\sigma\in\Hom_\mathrm{Pos}(\Spec(\mathbb{C}[[x]]),\mathrm{Wide}(\bK Q))$ such that $\sigma((0))=\leftangle\bK\to\bK\rightangle$ and $\sigma((x))=\mods(\bK A_2)$, one can check that $\alpha\left(\mathsf{filtgen}\leftangle\mathbb{C}[[x]]\to\mathbb{C}[[x]],\mods(\mathbb{C}A_2)\rightangle\right)=\leftcurly M\xrightarrow{f_{tor}}M'\mid M,M'\in\mods(\mathbb{C}[[x]])\rightcurly$.
\end{exmp}

\appendix
\section{t-structures of Dynkin algebras}\label{App}
This final section is devoted to the classification of the aisles of $\cD(\bK Q)$, for any field $\bK$ and Dynkin quiver $Q$. We show that the lattice $\Aisle(\cD(\bK Q))$ is isomorphic to the lattice $\mathrm{Filt}(\mathbf{Nc}(Q))$ of filtrations (i.e. non increasing sequences) of noncrossing partitions of $Q$, and so that it is independent of the field. In particular, this result specifies the classification in \cite{SR} to algebras of finite representation type.

\begin{lem}\label{ICE}
    Let $\cA$ be an hereditary abelian category and $\cS\subseteq\cD(\mathcal{A})$ a suspended subcategory. Then, for any map $f\in\Hom_\cA(H^n(\cS),H^{n+1}(\cS))$, it holds that $\ker f\in H^n(\cS)$ and $\coker f\in H^{n+1}(\cS)$. In particular,
    \[\ldots\supseteq H^n(\cS)\supseteq H^{n+1}(\cS)\supseteq\ldots\]
    is a filtration of subcategories closed under images, cokernels and extensions (ICE-closed).
\begin{proof}
    Let $f:A\to B$ with $A\in H^n(\cS)$ and $B\in H^{n+1}(\cS)$, then there is a triangle
    \[\begin{tikzcd} B[-n-1] \arrow[r] & \cone f[-n-1] \arrow[r] & A[-n] \arrow[r,"+"] & {} \end{tikzcd}\]
    where $\cone f[-n-1]\cong\ker f[-n]\oplus\coker f[-n-1]$. Since $\cS$ is closed under extensions and summands, $\cone f[-n-1]\in\cS$ and we get the statement. As for the last part, let $f:C\to D$ in $\Hom(H^n(\cS),H^n(\cS))$. Since $C$ lies also in $H^{n-1}(\cS)$, we have that $\coker f\in H^n(\cS)$ and $\ker f\in H^{n-1}(\cS)$, so $\im f\cong\coker(\ker f\hookrightarrow C)\in H^n(\cS)$. The closure under extensions of $H^n(\cS)$ follows from the one in $\cS$.
\end{proof}
\end{lem}

\begin{rem}\label{ITRem}
    We want to note that some results from \cite{IT} hold also for ICE-closed subcategories (of artinian, Krull-Schmidt, hereditary abelian categories, e.g. $\mods(\bK Q)$). We will call an ICE-closed subcategory $\cH$ finitely generated if it contains a finite set of indecomposable object $\mathcal{I}$, such that $\cH\subseteq\gen(\mathcal{I})$, where $\gen$ denotes the closure under quotients and finite direct sums. Note that it is always the case in $\mods(\bK Q)$ when $Q$ is Dynkin. In particular, the following holds:
    \begin{itemize}
        \item \cite[Theorem 2.8]{IT} A finitely generated ICE-closed subcategory $\cH$ has a minimal generator, unique up to isomorphism, which is the direct sum of all its indecomposable split projectives;
        \item \cite[Proposition 2.12]{IT} For any ICE-closed subcategory $\cH$, the class
        \[\alpha(\cH)=\leftcurly M\in\cH\mid\text{for all }f\in\Hom(\cH,M),\,\ker f\in\cH\rightcurly\]
        is a wide subcategory. (See \cite[Proposition 2.2]{Sak} for a proof);
        \item \cite[Theorem 2.15]{IT} For any finitely generated ICE-closed subcategory $\cH$, let
        \begin{equation*}
            \alpha_s(\cH)=\leftcurly
            \begin{aligned}
                M\in\cH\mid&\text{ for all surjections } f:Z\to M \text{ with}\\
                &\ Z\in\cH\text{ split projective},\,\ker f\in\cH
            \end{aligned}\rightcurly
        \end{equation*}
        Then, $\alpha(\cH)=\alpha_s(\cH)$.
    
        \noindent To make the proof work for $\cH$ note that, at the end of their proof, $\ker g'$ is not just a quotient of $\ker g'h$ but, since $h$ is surjective, it is the image of $h_{\mid\ker g'h}:\ker g'h\to Y'$. Since both $\ker g'h$ and $Y'$ are in $\cH$, so does $\ker g'$;
        \item \cite[Theorem 2.16]{IT} For any finitely generated ICE-closed subcategory $\cH$, any split projective $U$ of $\cH$ is in $\alpha_s(\cH)=\alpha(\cH)$. In particular, $\cH\subseteq\gen(\bigoplus U_i)\subseteq\gen(\alpha(\cH))$.
    \end{itemize}
\end{rem}

\begin{thm}\label{FiltWide}
    For any field $\bK$ and Dynkin quiver $Q$, the following assignments form a bijection
    \begin{equation*}\begin{gathered}
        \mathrm{Susp}(\cD^c(\bK Q))\longleftrightarrow\mathrm{Filt}(\mathrm{Wide}(\bK Q)) \\
        \cS\longmapsto\alpha(H^n(\cS)) \\
        \leftcurly X\in\cD^c(\bK Q)\mid H^n(X)\in\gen(\cW_n)\cap\cW_{n-1}\rightcurly\longmapsfrom\,\left(\ldots\supseteq\cW_n\supseteq\cW_{n+1}\supseteq\ldots\right)
    \end{gathered}\end{equation*}
\begin{proof}
    \textit{The first assignment is well defined.} By \cref{ITRem}, for any $n\in\bZ$ the class $\alpha(H^n(\cS))$ is a wide subcategory and, by \cref{ICE}, $H^{n+1}(\cS)\subseteq\alpha(H^n(\cS))$, thus $\alpha(H^{n+1}(\cS))\subseteq\alpha(H^n(\cS))$.
    
    \noindent\textit{The second assignment is well defined.} Denote by $\cS_\cW$ the image of the assignment. It is clearly closed under summands and positive shifts, so it remains to show that it is closed under extensions. Let $X,Y\in\cS_\cW$ and consider a distinguished triangle
    \[\begin{tikzcd} X \arrow[r] & Z \arrow[r] & Y \arrow[r,"+"] & {} \end{tikzcd}\]
    It induces the exact sequence in cohomology
    \[\begin{tikzcd}
    H^{n-1}(Y) \arrow[r,"f"] & H^n(X) \arrow[r] & H^n(Z) \arrow[r] & H^n(Y) \arrow[r,"g"] & H^{n+1}(X)
    \end{tikzcd}\]
    from which we obtain the short exact sequence
    \[\begin{tikzcd} 0 \arrow[r] & \coker f \arrow[r]  & H^n(Z) \arrow[r] & \ker g \arrow[r] & 0 \end{tikzcd}\]
    Since $H^n(X)$ lies in $\gen(\cW_n)\cap\cW_{n-1}$, on one hand, $\coker f$ is in $\gen(\cW_n)$ as well, on the other hand, there exist $G\in\cW_{n-1}$ and a surjection $\pi:G\to H^{n-1}(Y)$ such that $f\circ\pi\in\Hom(\cW_{n-1},\cW_{n-1})$ and $\coker f=\coker(f\circ\pi_{n-1})$ is in $\cW_{n-1}$. Moreover, we will prove later that $\cW_n\subseteq\alpha(\gen(\cW_n)\cap\cW_{n-1})$ thus $\ker g$ lies in $\gen(\cW_n)\cap\cW_{n-1}$. Since, by \cite[Proposition 2.13]{IT}, $\gen(\cW_n)$ is closed under extensions, it follows that $H^n(Z)$ lies in $\gen(\cW_n)\cap\cW_{n-1}$.
    
    \noindent\textit{The identity on the left.} Let $\cS\subseteq\cD^c(\bK Q)$ be a suspended subcategory and
    \[\cS_\alpha:=\leftcurly X\in\cD^c(\bK Q)\mid H^n(X)\in\gen(\alpha(H^n(\cS)))\cap\alpha(H^{n-1}(\cS))\rightcurly\]
    We will show that $\cS=\cS_\alpha$. For any $n\in\bZ$ we have that, by \cref{ITRem}, $H^n(\cS)\subseteq\gen(\alpha(H^n(\cS)))$ and, by \cref{ICE}, $H^n(\cS)\subseteq\alpha(H^{n-1}(\cS))$. Thus, $\cS\subseteq\cS_\alpha$. For the other inclusion, let $X\in\cS_\alpha$, then for any $n\in\bZ$ there is $G_n\in\alpha(H^n(\cS))$ and a surjective map $\pi:G_n\twoheadrightarrow H^n(X)$. In particular, by definition of $\alpha$, $\ker\pi\in H^n(\cS)\subseteq\alpha(H^{n-1}(\cS))$. Consider the distinguished triangle
    \[\begin{tikzcd} {G_n[-n]} \arrow[r,"{\pi[-n]}"] & {H^n(X)[-n]} \arrow[r] & {\ker\pi[-n+1]} \arrow[r,"+"] & {} \end{tikzcd}\]
    Since $\cS$ is closed under extensions, $H^n(X)[-n]$ lies in $\cS$ for any $n\in\bZ$, so $X\cong\bigoplus H^n(X)[-n]$ does.
    
    \noindent\textit{The identity on the right.} It suffices to prove that for any filtration of wide subcategories
    \[\ldots\supseteq\cW_n\supseteq\cW_{n+1}\supseteq\ldots\]
    we have the equality $\cW_n=\alpha\left(\gen(\cW_n)\cap\cW_{n-1}\right)$. For the inclusion from left to right, we have that $\cW_n\subseteq\gen(\cW_n)\cap\cW_{n-1}$ and for any $W\in\cW_n$ and $f:A\to W$ with $A\in\gen(\cW_n)\cap\cW_{n-1}$, there is $G\in\cW_n$ and a surjection $\pi:G\to A$ such that, by surjectivity of $\pi$, $\ker f=\im\left(\ker(f\circ\pi)\xrightarrow{\pi}A\right)$. In particular, $\ker f$ is both a quotient of $\ker(f\circ\pi)$, so it is in $\gen(\cW_n)$, and the image of a morphism in $\cW_{n-1}$, so it is in $\cW_{n-1}$. For the other inclusion, let $M\in\alpha\left(\gen(\cW_n)\cap\cW_{n-1}\right)$, in particular there are $G'\in\cW_n$ and a surjection $\pi':G'\to M$ such that, by definition of $\alpha$, $\ker\pi'\in\gen(\cW_n)\cap\cW_{n-1}$. Thus, there exist $G''\in\cW_{n}$ and a surjection $\pi'':G''\to\ker\pi'$ such that $M\cong\coker\left(\ker\pi'\xrightarrow{i}G'\right)\cong\coker\left(G''\xrightarrow{i\circ\pi''}G'\right)$ and so it lies in $\cW_n$.
\end{proof}
\end{thm}

\begin{cor}\label{FiltNc}
    For any field $\bK$ and Dynkin quiver $Q$, there is a bijection
    \[\Aisle(\cD(\bK Q))\longleftrightarrow\mathrm{Filt}(\mathbf{Nc}(Q))\]
\begin{proof}
    Note that by \cite[Theorem 4.5 (i)]{SP}, there is a bijection between $\Aisle(\cD(\bK Q))$ and $\mathrm{Susp}(\cD^c(\bK Q))$. Then, by \cref{FiltWide}, this further specifies to a bijection with $\mathrm{Filt}(\mathrm{Wide}(\bK Q))$. Finally, by \cite[Theorem 1.1]{IT}, the lattice of wide subcategories of $\mods(\bK Q)$ is isomorphic to the lattice of noncrossing partitions of $Q$, thus we get the result.
\end{proof}
\end{cor}

\begin{exmp}
    The aisle of $\cD(\bK A_3)$ corresponding to the filtration
    \[\cW_1=0\subseteq\add(P_1)\subseteq\add(P_2,P_1,E_1)\subseteq\mods(\bK A_3)=\cW_{-2}\]
    is the one generated by the modules in red in the figure below
    \begin{figure}[ht]
        \centering
        \includegraphics[width=15cm, height=2.5cm]{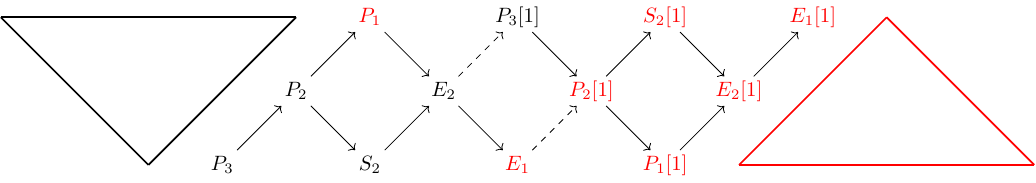}
    \end{figure}
\end{exmp}

\bibliographystyle{amsalpha}
\bibliography{references}
\end{document}